\documentclass[11pt]{amsart}
\usepackage{amssymb,amsthm,amsmath,amstext}
\usepackage{mathrsfs}  
\usepackage{bm}        
\usepackage{mathtools} 
\usepackage{color}
\usepackage[bbgreekl]{mathbbol}

\usepackage{cite}

\usepackage[left=1.5in,top=1in,right=1.5in,bottom=1in]{geometry}

\usepackage{graphicx}

\usepackage[all]{xy}
\usepackage{tikz}
\usepackage{tikz-cd}
\usepackage{enumitem}

\theoremstyle{plain}
\newtheorem{theorem}{Theorem}[section]

\newtheorem{proposition}[theorem]{Proposition}
\newtheorem{lemma}[theorem]{Lemma}
\newtheorem{corollary}[theorem]{Corollary}

\newtheorem{assumption}[theorem]{Assumption}

\theoremstyle{definition}
\newtheorem{definition}[theorem]{Definition}
\newtheorem{notation}[theorem]{Notation}

\theoremstyle{remark}
\newtheorem{remark}[theorem]{Remark}

\newcommand{\sheaf}[1]{\mathscr{#1}}
\newcommand{\PP}{\sheaf{P}}

\newcommand{\LL}{\sheaf{L}}
\newcommand{\OO}{\sheaf{O}}

\newcommand{\EE}{\sheaf{E}}
\newcommand{\FF}{\sheaf{F}}
\newcommand{\GG}{\sheaf{G}}
\newcommand{\II}{\sheaf{I}}
\newcommand{\HH}{\sheaf{H}}
\newcommand{\NN}{\sheaf{N}}

\newcommand{\ZZ}{\sheaf{Z}}
\newcommand{\WW}{\sheaf{W}}
\newcommand{\BB}{\sheaf{B}}
\newcommand{\CC}{\sheaf{C}}
\newcommand{\YY}{\sheaf{Y}}
\newcommand{\XX}{\sheaf{X}}
\newcommand{\sA}{\sheaf{A}}

\newcommand{\UU}{\sheaf{U}}

\newcommand{\F}{\mathbb F}
\newcommand{\Z}{\mathbb Z}

\renewcommand{\P}{\mathbb P}

\DeclareMathOperator{\grad}{grad}

\usepackage[backref=page]{hyperref}

\begin{document}

\title[Explicit bounds on common projective torsion points]{Explicit bounds on common projective torsion points of elliptic curves}

\author[B\"ohning]{Christian B\"ohning}
\address{Christian B\"ohning, Mathematics Institute, University of Warwick\\
Coventry CV4 7AL, England}
\email{C.Boehning@warwick.ac.uk}

\author[Bothmer]{Hans-Christian Graf von Bothmer}
\address{Hans-Christian Graf von Bothmer, Fachbereich Mathematik der Universit\"at Hamburg\\
Bundesstra\ss e 55\\
20146 Hamburg, Germany}
\email{hans.christian.v.bothmer@uni-hamburg.de}

\author[Hubbard]{David Hubbard}
\address{David Hubbard, Mathematics Institute, University of Warwick\\
Coventry CV4 7AL, England}
\email{David.Hubbard@warwick.ac.uk}


\begin{abstract}
Suppose $E_1, E_2$ are elliptic curves (over the complex numbers) together with double coverings $\pi_i \colon E_i \to \P^1$ ramified in the two-torsion points of $E_i$. Let $E_i[\infty]$ be the torsion points on $E_i$. In \cite{BFT18}, Bogomolov, Fu and Tschinkel ask if the number of points in $\pi_1 (E_1[\infty]) \cap \pi_2 (E_2[\infty])$ is uniformly bounded in the case when the branch loci of the $\pi_i$ do not coincide.  Very recently this was answered affirmatively \cite{DKY20, Kueh21, Gao21, DGH21, GGK21} and also \cite{Poi22-1, Poi22-2}, but realistic effective bounds are unknown. 

In this article we obtain such bounds for common projective torsion points on elliptic curves under some mild extra assumptions on the reduction type of the input data at given primes. The method is based on Raynaud's original groundbreaking work on the Manin-Mumford conjecture \cite{Ray83-1, Ray83-2}. In particular, we generalise several of his results to cases of bad reduction using techniques from logarithmic algebraic geometry. 
\end{abstract}

\maketitle

\tableofcontents

\section{Introduction and basic setup}\label{sIntroduction}

The following question in the theory of unlikely intersections, which was raised in \cite{BFT18} and is closely related to the uniform Manin-Mumford conjecture, has recently attracted a lot of attention: suppose $E_i$, $i=1,2$ are elliptic curves over the complex numbers (=one-dimensional abelian varieties), together with \emph{standard projections} to $\P^1$, $\pi_i \colon E_i \to \P^1$. Here and in the following by a standard projection we will mean a degree $2$ morphism $\pi_i \colon E_i \to \P^1$ that identifies each point on $E_i$ with its inverse, hence is ramified in the four $2$-torsion points $E_i[2]$ of $E_i$. Suppose furthermore that the branch points in $\P^1$ of these two double covers do not coincide as subsets of points in $\P^1$. Write $E_i[\infty]$ for the torsion points on $E_i$ (of arbitrary order). What is the smallest $C$ such that under these hypotheses one can conclude
\[
\left| \pi_1(E_1[\infty]) \cap \pi_2 (E_2[\infty]) \right| \le C \quad ?
\]
It is not too difficult to deduce that for any given $(E_i, \pi_i)$, the set $\pi_1(E_1[\infty]) \cap \pi_2 (E_2[\infty])$ is finite. In fact, this already follows from Raynaud's result \cite{Ray83-1} that the torsion points of a complex abelian variety $A$ that lie on some curve $C\subset A$ that is not elliptic are finite in number. Indeed, consider the four to one covering
\[
\pi_1\times \pi_2 \colon E_1 \times E_2 \to \P^1 \times \P^1
\]
and the preimage of the diagonal 
\[
X = (\pi_1\times \pi_2)^{-1} (\Delta ). 
\]
Then it is easy to see that under the assumption that the sets of branch points of $\pi_1, \pi_2$ do not coincide, this curve is irreducible and not elliptic. 

\medskip

Recently, several authors \cite{DKY20, Kueh21, Gao21, DGH21, GGK21} finally managed to show, as a corollary of their work, that one can choose one constant $C$ that works for all pairs $(E_i, \pi_i)$ above at once (i.e., \emph{uniformity} holds). Poineau in \cite{Poi22-1, Poi22-2} also proved this using a different technique using Berkovich spaces over the integers and dynamics of Latt\`{e}s maps. 

To the best of our knowledge, these approaches have so far failed to determine the minimal possible $C$ above and not yielded \emph{effective realistic} bounds. However, one knows pairs $(E_i, \pi_i)$ where $|\pi_1(E_1[\infty]) \cap \pi_2 (E_2[\infty])|$ is comparatively large \cite{BF17,FS19}. The current record (in \cite{FS19}), as far as we are aware, is $34$. 

\medskip

In this work, we propose to obtain such effective realistic bounds for the common torsion points $\pi_1(E_1[\infty]) \cap \pi_2 (E_2[\infty])$, or some large subset of this set, under some mild extra assumptions on the curves $E_i$, taking our point of departure from the methods used by Raynaud's in \cite{Ray83-1}, \cite{Ray83-2}. 

In particular, we generalise several arguments by Raynaud to the log smooth setting. 


\medskip

The road map of the paper is as follows: in Section \ref{sGoodReduction} we obtain explicit bounds on common projective torsion points of order coprime to $p$ for two elliptic curves together with standard projections that have good reduction at a given place of some number field lying over a given prime $p$. This is the content of Theorem \ref{tCoarseBounds}. 

We refine these bounds in Section \ref{sRefinements} in Proposition \ref{pBothSupersingular} and Proposition \ref{pRefinementFrobLiftability}. 

In Section \ref{pBoundsTwoPrimes} we show how one can obtain explicit bounds on common projective torsion points for curves with good reduction at two given primes. 

In Section \ref{sGoodMult}, we generalise the preceding results to the case when one or two of the elliptic curves are allowed to have bad multiplicative reduction at a given place. This is done in Theorem \ref{tMainMixedReductionTypes}. This is valid under Assumption \ref{aGoodBadReduction}. Part b) of that Assumption is less geometric, but we expect it to be implied by part a). In fact, we show that this is true in special cases in Theorem \ref{tFinitenessGeneralised}. The proof is longer and occupies the remaining sections of the paper. It involves ideas from logarithmic algebraic geometry and generalises an argument in \cite{Ray83-2}. 

\begin{remark}\label{rReductionToNumberFields}
To obtain effective bounds of the type mentioned above, it is no essential restriction to assume that both $E_1$ and $E_2$ are defined over a number field; indeed, if $E_i, \pi_i$ are initially defined over $\mathbb{C}$, there exists a $\mathbb{Z}$-algebra $A$ of finite type contained in $\mathbb{C}$ such that all these data are already defined over $S =\mathrm{Spec}\, (A)$. Replacing $S$ by some nonempty open subset of necessary, we can assume that there are 
\begin{enumerate}
\item
One-dimensional abelian schemes $\EE_i \to S$ with geometric generic fibres the $E_1$, with morphisms $\pi_{i, S} \colon \EE_i \to \P^1_S$ that are standard projections on each geometric fibre and the given standard projections on the geometric generic fibres.
\item
The scheme $\XX = (\pi_{1, S} \times \pi_{2, S} )^{-1} (\Delta_{\P^1_S \times_S \P^1_S} )$ is a proper flat $S$-curve with geometric generic fibre $X$. 
\item
For each geometric fibre, the set of common branch points of the standard projections has the same cardinality as on the geometric generic fibre. 
\end{enumerate}
Let $s$ be a closed point of $S$ lying above the generic point of $\mathrm{Spec}\, \Z$. The number of torsion points of the geometric generic fibre $E_1\times E_2$ that are contained in $X$ specialise injectively (since we are in equal characteristic zero) to torsion points of $\EE_{1, s}\times \EE_{2,s}$ lying on $\XX_s$. In any case, if $t_{E_1\times E_2, X}$ denotes the number of torsion points of $E_1\times E_2$ that are contained in $X$, then 
\[
\left| \pi_1(E_1[\infty]) \cap \pi_2 (E_2[\infty]) \right| \le \frac{t_{E_1\times E_2, X}}{4} + 8
\]
(since the covering $\pi_1\times \pi_2 \colon X  \to \Delta \simeq \P^1$ is \'{e}tale of degree $4$ away from the points that coincide with one of the branch points of $\pi_1$ or $\pi_2$, which are at most eight). Thus a bound on the number of torsion points of $\EE_{1, s}\times \EE_{2,s}$ lying on $\XX_s$ will in general yield a very good bound for our original problem. 
\end{remark}

In view of the preceding remark, we usually assume in the sequel that the data $E_i, \pi_i$ is defined over some number field $K$, with ring of integers $\mathcal{O}_K$. In that case, using the same spread construction as in Remark \ref{rReductionToNumberFields}, we can assume that $E_i$ extend to abelian $U$-schemes for some nonempty open subset $U\subset \mathrm{Spec}\, \mathcal{O}_K$, and the standard projections $\pi_i$ extend to $U$-morphisms that induce standard projections on every geometric fibre, with the number of common branch points of the standard projections being constant in the family. Moreover, we can assume $U$ is unramified over $\mathrm{Spec}\, \mathbb{Z}$. We can then choose a closed point $v$ of $U$ lying over a prime $p$, and identifying $v$ with the corresponding extension of the $p$-adic valuation to $K$, we can pass to the completion of the maximal unramified extension
\[
\widehat{K_v^{\mathrm{ur}}}
\]
with valuation ring $R \supset \mathcal{O}_{K, v}$ isomorphic to the ring of Witt vectors $W(\overline{\mathbb{F}}_p)$ with coefficients in the algebraic closure of the finite field $\mathbb{F}_p$. This shows that there are always plenty of prime numbers $p$ satisfying the following assumptions. 

\begin{assumption}\label{aBasic}
There exists a prime $p$ and a place $v$ of $K$ unramified over $p$ with the following properties. Let $R = W(\overline{\mathbb{F}}_p)$ be the Witt vectors over $k:=\overline{\mathbb{F}}_p$ with fraction field $F= \widehat{K^{\mathrm{ur}}_v} \supset K$. 
\begin{enumerate}
\item
There are abelian schemes 
\[
 \EE_i \to \mathrm{Spec}\, R , \quad i=1, 2
\]
with geometric generic fibres equal to (the base change to $F$) of the given elliptic curves $E_i$ defined over $K$. 
\item
For $i=1, 2$, there are $R$-morphisms
\[
\xymatrix{
\EE_i  \ar[rd]  \ar[rr]^{\pi_{i, R}} &  & \P^1_R \ar[ld] \\
 & \mathrm{Spec}\, R & 
}
\]
inducing standard projections on every geometric fibre and the given standard projections (base-changed to $F$) on the geometric generic fibre. Sometimes, by slight abuse of notation, we will also simply write $\pi_i$ for $\pi_{i, R}$ if there is no risk of confusion.
\item
The number of common branch points in $\P^1$ of the $\pi_{i, R}$ is the same on the special fibre as on the geometric generic fibre. 
We write 
\[
\XX = \left( \pi_{1, R} \times \pi_{2, R} \right)^{-1} \left( \Delta_{\P^1_R \times_R \P^1_R} \right) 
\]
for the preimage of the diagonal, which is a proper flat $R$-curve. 
\end{enumerate}
\end{assumption}

Of course, the prime numbers in question depend on the data $(E_i, \pi_i)$ and can be very large in special cases.

\section{Bounds on common torsion points if both curves have good reduction}\label{sGoodReduction}

Here we assume we are given elliptic curves $E_i$ and standard projections $\pi_i$, $i=1, 2$, defined over a number field $K$ satisfying Assumption \ref{aBasic} above, and we wish to show how a method pioneered by Raynaud in \cite{Ray83-1} yields very realistic bounds on 
\[
\left| \pi_1\left( E_1[\infty]^{(p')} \right) \cap \pi_2 \left( E_2[\infty]^{(p')} \right) \right| 
\]
where we denote by $E_i[\infty]^{(p')}$ the coprime to $p$ torsion on $E_i$. 

\medskip

We denote the abelian $R$-scheme $\EE_1 \times_{\mathrm{Spec}\, R}\EE_2$ by $\sA$. We denote its special fibre over $\mathrm{Spec}\, k$ by $\sA_0$. 

\begin{lemma}\label{lDefCoprimeTorsion}
All torsion points in $(E_1\times E_2)(\bar{K})$ of order not divisible by $p$ are defined over $K_v^{\mathrm{ur}}$, hence can be viewed as sections of $ \sA \to \mathrm{Spec}\, R$. 
Moreover, the reduction map $\sA (\mathrm{Spec}\, R) \to \sA_0(k)$ gives an isomorphism from the $n$-torsion points in $(E_1\times E_2)(\bar{K})$ onto the $n$-torsion points of $\sA_0(k)$ as long as $p$ does not divide $n$. 
\end{lemma}

\begin{proof}
Indeed, given an abelian scheme over a discrete valuation ring of mixed characteristic $(0,p)$, the sub-group scheme of $n$-torsion points is finite and \'{e}tale over the base provided $p$ does not divide $n$ \cite[Prop. 1.34]{Sai13}. 
\end{proof}

\begin{lemma}\label{lInImageOfMultByP}
If $p$ does not divide $n$, every $n$-torsion point in $(E_1\times E_2)(\bar{K})$ can be written as $p$-times another such $n$-torsion point. Thus every section of $ \sA \to \mathrm{Spec}\, R$ corresponding to such a torsion point is in the image of another $R$-valued point in $\sA$ under the multiplication by $p$ map $[p] \colon \sA \to \sA$ on the abelian scheme $\sA$. 
\end{lemma}

\begin{proof}
This is simply because multiplication by $p$ is an isomorphism on $\Z/n\times \Z/n$. 
\end{proof}

Taken together, these two lemmas directly imply

\begin{proposition}\label{pSimpleBounds}
A bound on
\[
t (E_1, \pi_1, E_2, \pi_2, p'):= \left| \pi_1\left( E_1[\infty]^{(p')} \right) \cap \pi_2 \left( E_2[\infty]^{(p')} \right) \right| 
\]
is given by
\[
\frac{1}{4} \left| \mathrm{im}   \left( p \sA (R) \cap \XX (R) \to \XX_0(k) \right) \right|  + 8
\]
where the arrow in the displayed formula is the specialisation map and $\XX_0$ denotes the central fibre of the curve $\XX$. 

Moreover, putting $R_1= R/p^2$, and denoting by $\sA_1$, $\XX_1$ the base change of $\sA, \XX$ to $\mathrm{Spec}\, R_1$, a bound on $t (E_1, \pi_1, E_2, \pi_2, p')$ is also obtained by 
\[
\frac{1}{4} \left| \mathrm{im}   \left( p \sA_1 (R_1) \cap \XX_1 (R_1) \to \XX_0(k) \right) \right|  + 8
\]
\end{proposition}

To explain the key ideas in a simple context, in the sequel of this section we will, in addition to Assumption \ref{aBasic}, make the following 

\begin{assumption}\label{aSimplest}
The branch loci of the standard projections on the special fibres $(\pi_1)_k \colon \EE_{1}\times_{\mathrm{Spec}\, R} \mathrm{Spec} k \to \P^1_k$ and $(\pi_2)_k \colon \EE_{2}\times_{\mathrm{Spec}\, R} \mathrm{Spec} k \to \P^1_k$ are disjoint. (Hence the same is true for the generic fibres). 
\end{assumption}

This implies that $\XX$ is a smooth $R$-curve. 

\begin{theorem}\label{tCoarseBounds}
Let $(E_1, \pi_1)$, $(E_2, \pi_2)$ satisfy Assumptions \ref{aBasic} and \ref{aSimplest}. Then 
\[
t (E_1, \pi_1, E_2, \pi_2, p') \le 2p^3 + 8. 
\]
\end{theorem}

\begin{proof}
By Proposition \ref{pSimpleBounds}, it suffices to show 
\[
 \left| \mathrm{im}   \left( p \sA_1 (R_1) \cap \XX_1 (R_1) \to \XX_0(k) \right) \right| \le 8p^3.
 \]
 For this, it is convenient, following ideas in \cite{Ray83-1}, to pass to a structure defined over $k$ to encode information about first-order infinitesimal deformations. We recall how this is done, following  \cite{Ray83-1}: writing $\sA_1 (R_1, \XX_0)$ for the set of $R_1$-points of $\sA_1\to \mathrm{Spec}\, R_1$ that specialise to a point in $\XX_0$, we note that there is a factorisation 
\[
\xymatrix{
\sA_1 (R_1, \XX_0) \ar[r]^{\tau} \ar[rd]_{\mathrm{specialization}}  & V_0 (k)\ar[d]^{\varpi} \\
 & \XX_0(k)
}
\]
where $V_0 \to \XX_0$ is a certain \emph{affine bundle} over $X_0$, obtained as follows: consider the normal bundle $\NN_{\XX_0/\sA}$ with subbundle $\NN_{\XX_0/\sA_0}$ and form 
\[
V_0 = \P \bigl( \NN_{\XX_0/\sA} \bigr) \backslash \P \bigl( \NN_{\XX_0/\sA_0} \bigr)
\]
which is naturally an affine bundle over $\XX_0$. Since sections of $\sA_1\to \mathrm{Spec}\, R_1$ that specialise to a point $x$ in $\XX_0$ have a normal direction at $x$ that is not contained in $\AA_0$, we get a factorisation as claimed in the diagram above. Write $\XX'_0$ for the curve in $V_0$ whose $k$-points are the image of $\XX_1(R_1)$ under $\tau$. It lies isomorphically over $\XX_0$ via $\varpi$. Now if $f\colon \sA_1\to \sA_1$ is any $R_1$-morphism whose base change to the central fibre has zero differential (such as, for example, the multiplication by $p$ map), we get a factorization
\[
\xymatrix{
\sA_1(R_1) \ar[rr]^f \ar[rd] &  & \sA_1 (R_1)\\
 & \sA_0(k) \ar[ru]  & 
}
\]
Denote by $\YY_0 \subset \sA_0$ the \emph{reduced} preimage of $\XX_0$ under the multiplication by $p$ map on $\sA_0$. 
Then, in particular, all points in $\YY_0(k)$ give in this way unique points in $p \sA_1 (R_1)$ which we can specialise again to $V_0(k)$: it turns out, \cite[Prop. 3.3.1]{Ray83-1}, that the resulting set $\YY_0'(k) \subset A_0(k)$ is the set of $k$-points of another projective curve $\YY_0'\subset V_0$, and we are interested in computing the intersection number $\XX_0'.\YY_0'$ in $V_0$ (or better its compactification $\P \bigl( \NN_{\XX_0/\sA} \bigr) = \P \bigl( \NN_{\XX_0/\sA_0}\oplus \NN_{\XX_0/\XX}\bigr) = \P \bigl( \NN_{\XX_0/\sA_0}\oplus \OO_{\XX_0}\bigr)$). 
Indeed, $k$-points of $\XX_0'.\YY_0'$ map, by construction, surjectively onto the set $\mathrm{im}   \left( p \sA_1 (R_1) \cap \XX_1 (R_1) \to \XX_0(k) \right)$ whose cardinality we are trying to bound. 

Here it is also essential to notice that $\XX_0'\cap \YY_0'$ is actually finite under Assumption \ref{aSimplest}; indeed, under that assumption, $\XX_0'$ is irreducible of genus $5$, in particular, does not contain any elliptic component. Then the finiteness follows from \cite[proof of Thm. 4.4.1]{Ray83-1} as well as in this special case \cite{Ray83-2}. 

\medskip

The Picard group of the projective bundle $ \P \bigl( \NN_{\XX_0/\sA_0}\oplus \OO_{\XX_0}\bigr)$ can be generated by the zero section $\XX_0'=\P \bigl( \OO_{\XX_0}\bigr)$ and the class of a fibre, and since $\YY_0'$ is contained in the finite part (the complement of the infinity section), intersecting with the infinity section tells us that $\YY_0'$ is a multiple of $\XX_0'$. Moreover, intersecting with a fibre, we see that 
\[
\YY_0' \equiv \delta \XX_0'
\]
where $\delta$ is the degree of $\varpi \colon \YY_0' \to \XX_0$. Since the normal bundle of $\XX_0'$ in $\P \bigl( \NN_{\XX_0/\sA_0}\oplus \OO_{\XX_0}\bigr)$ is nothing but $\NN_{\XX_0/\sA_0}$ we get 
\[
\XX_0'.\YY_0' = \delta (\XX_0.\XX_0)_{\sA_0}= 8 \delta . 
\]
Recall that $\XX_0$ is the preimage in $\sA_0 \times \sA_0$ of the diagonal in $\P^1_k\times \P^1_k$ under a $4:1$ covering map, whence the factor $8$ in the preceding formula. 

\medskip

Thus to finish the proof we need to bound $\delta = \deg \varpi$, more precisely, we need to show 
\[
\delta \le p^3. 
\]
For this, remark that by construction there is a commutative diagram
\[
\xymatrix{
\YY_0 \ar[rd]_{(\cdot p)\mid_{\YY_0}} \ar[r]^{\theta} & \YY_0' \ar[d]^{\varpi}\\
 & \XX_0
}
\]
Thus $\delta$ is bounded from above by the degree of $(\cdot p)\mid_{\YY_0}$. We will show that this latter is less than or equal to $p^3$. Indeed, the multiplication by $p$-map on the abelian surface $\sA_0$ has degree $p^4$, but it factors over the relative Frobenius. Since $\YY_0$ is defined to be the reduced preimage of $\XX_0$ under this map, we get the desired bound. 
\end{proof}

\section{Refinements according to the reduction type and Frobenius liftability}\label{sRefinements}

In certain case, the bounds obtained in Theorem \ref{tCoarseBounds} can be substantially refined. First recall \cite[V.3 Theorem 3.1]{Sil09}

\begin{definition}\label{dSupersingularOrdinary}
Let $E_0$ be a curve over $k=\overline{\mathbb{F}}_p$, and denote by $E_0(k)[p]$ the group that is the kernel of the multiplication by $p$-map $E_0(k) \to E_0(k)$. Then $E_0$ is called \emph{ordinary} if $E_0(k)[p]\simeq \Z/p\Z$ and \emph{supersingular} if $E_0(k)[p]= \{0\}$. 
\end{definition}

If $E_0$ is an elliptic curve over $k$, then $E$ is ordinary if and only if one has a factorisation 
\[
\xymatrix{
E \ar[rd]^{\mathrm{Fr}} \ar[rr]^{\cdot p} & & E \\
 & E' \ar[ru]^{g} & 
 }
\]
where $\mathrm{Fr}$ is the relative Frobenius of degree $p$, $E'$ the Frobenius twist of $E$, and $g$ is \'{e}tale of degree $p$. An elliptic curve $E_0$ as above is supersingular if and only if the multiplication by $p$ map factors as 
\[
\xymatrix{
E \ar[rd]^{\mathrm{Fr}^2} \ar[rr]^{\cdot p} & & E \\
 & E'' \ar[ru]^{g} & 
 }
\]
and $g$ is an isomorphism. 

\begin{proposition}\label{pBothSupersingular}
Keeping all the assumptions of Theorem \ref{tCoarseBounds} and assuming in addition that the reductions $\EE_{1, 0}$ and $\EE_{2, 0}$ of the curves $E_1$ and $E_2$ are both supersingular, we have 
\[
t (E_1, \pi_1, E_2, \pi_2, p') \le 2p^2 + 8. 
\]
\end{proposition}

\begin{proof}
Indeed, this will follow if we can show that the quantity $\delta$ appearing in the proof of Theorem \ref{tCoarseBounds} is bounded by $p^2$ in this case. It suffices to show that this is so for the degree of $(\cdot p)\mid_{\YY_0}\colon \YY_0 \to \XX_0$. In this case, by definition, $\YY_0$ is isomorphic to the second Frobenius twist of $\XX_0$ and $(\cdot p)\mid_{\YY_0}$ the second power of Frobenius, so it has degree $p^2$.
\end{proof}

As the previous proof illustrates, improving the bounds is closely connected to improving the bounds on $\delta = \deg \varpi$. Under certain conditions, one can get such better bounds also in the case when $\sA_0$ is ordinary. So we will now consider the case when $E_1, E_2$ have good ordinary reduction. 

The idea is to look at the connected-\'{e}tale exact sequence for the finite flat group scheme of $p$-torsion points $\GG_p$ on $\sA_1\to \mathrm{Spec}\, R_1$
\[
\xymatrix{
0 \ar[r] & \GG_p^0 \ar[r] & \GG_p \ar[r] & \GG_p^{et} \ar[r] & 0
}
\]
(here $\GG_p^0 \simeq \bbmu_p\times \bbmu_p$ and $\GG_p^{et} \simeq (\Z/p\Z)^2$ under the assumption that both $\EE_{i, 0}=\EE_i\times_{\mathrm{Spec}\, R} \mathrm{Spec}(k)$ are ordinary elliptic curves). If $\HH\subset \GG_p^{et}$ is an \'{e}tale subgroup-scheme over which the preceding sequence splits, i.e. if there exists a subgroup-scheme $\widetilde{\HH}$ of $\GG_p $ mapping isomorphically onto $\HH$, then the multiplication by $p$-map factors
\[
\xymatrix{
\sA_1 \ar[rr]_{\cdot p} \ar[rd]^q & & \sA_1\\
 & \BB=\sA_1/\widetilde{\HH} \ar[ru]^r & 
}
\]
where $q$ is \'{e}tale and $r$ restricted to the central fibre has differential zero, whence letting $\ZZ_0$ be the reduced preimage of $\XX_0$ under the map
\[
\BB_0 = \BB\times_{\mathrm{Spec}\, R_1} \mathrm{Spec}\, k \to \sA_1 \times_{\mathrm{Spec}\, R_1} \mathrm{Spec}\, k =\sA_0
\]
we get a factorisation of $\theta= \theta''\circ \theta'$
\[
\xymatrix{
\YY_0 \ar[rrd]_{(\cdot p)\mid_{\YY_0}} \ar[r]^{\theta'} & \ZZ_0 \ar[r]^{\theta''} \ar[rd]& \YY_0' \ar[d]^{\varpi}\\
 & & \XX_0 
}
\]
Here $\theta'$ has degree equal to the degree/order of $\widetilde{\HH}$. Thus, in this case, $\delta$, the degree of $\varpi$ is bounded by 
\[
\frac{p^3}{| \widetilde{\HH} |}. 
\]

\begin{proposition}\label{pRefinementFrobLiftability}
Keep the assumptions of Theorem \ref{tCoarseBounds} and assume in addition that the reductions $\EE_{1, 0}$ and $\EE_{2, 0}$ of the curves $E_1$ and $E_2$ are both ordinary. 
Suppose that the connected-\'{e}tale exact sequence
\[
\xymatrix{
0 \ar[r] & \bbmu_p \ar[r] & (\EE_i \times_{\mathrm{Spec}\, R} \mathrm{Spec}\, R_1 )_p \ar[r] & \Z/p\Z \ar[r] & 0
}
\]
for the finite flat group scheme of $p$-torsion points $(\EE_i\times_{\mathrm{Spec}\, R} \mathrm{Spec}\, R_1)_p$ splits for \emph{one of the curves} $\EE_i$. Then 
\[
t (E_1, \pi_1, E_2, \pi_2, p') \le 2p^2 + 8. 
\]
If this sequence splits for \emph{both curves} we have 
\[
t (E_1, \pi_1, E_2, \pi_2, p') \le 2p + 8. 
\]
\end{proposition}

\begin{proof}
This is immediate from the preceding reasoning and the proof of Theorem \ref{tCoarseBounds}. 
\end{proof}

Hence it becomes interesting to ascertain when, given elliptic curves $\EE_i\times_{\mathrm{Spec}\, R} \mathrm{Spec}\, R_1 \to \mathrm{Spec}\, R_1$ with ordinary reduction, the connected-\'{e}tale exact sequence for the finite flat group scheme of $p$-torsion points splits.

\medskip

Recall that $R_1=W_2(k)$, $W_2(k) = W(k)/p^2$, and that, as a \emph{set} $W_2(k) = k \times k$ with addition and multiplication defined explicitly by
\begin{align*}
(x_0,x_1)+ (y_0,y_1) &:= \bigl( x_0+y_0, x_1+y_1 - \sum_{i=1}^{p-1} \frac{(p-1)!}{i! (p-i)!} x_0^i y_0^{p-i} \bigr)\\
 & =\bigl( x_0+y_0, x_1+y_1 - \frac{(x_0+y_0)^p -x_0^p-y_0^p}{p} \bigr) \\
(x_0,x_1)\cdot (y_0,y_1) &:= \bigl( x_0y_0 , x_0^py_1 + y_0^px_1 + px_1y_1 \bigr) = \bigl( x_0y_0 , x_0^py_1 + y_0^px_1 \bigr).
\end{align*}
(The formulas defining addition and multiplication work more generally for any ring $A$ to give $W_2 (A)$). 

The Frobenius induces a homomorphism
\[
\mathrm{Fr}\colon W_2 (k) \to W_2 (k), \quad (x_0, x_1) \mapsto (x_0^p, x_1^p). 
\]

\begin{lemma}\label{lSplittingPTorsion}
Let $\EE_1 \to \mathrm{Spec}\, R_1$ be an elliptic curve with ordinary reduction $E_0/k$. The following are equivalent:
\begin{enumerate}
\item
the connected-\'{e}tale exact sequence
\[
\xymatrix{
0 \ar[r] & \bbmu_p \ar[r] & (\EE_1)_p \ar[r] & \Z/p\Z \ar[r] & 0
}
\]
for the finite flat group scheme of $p$-torsion points $(\EE_1)_p$ splits. 
\item
The (relative) Frobenius morphism $\mathrm{Fr}\colon E_0 \to E_0'$ lifts to a morphism
\[
\xymatrix{
F\colon \EE_1 \ar[rr]\ar[rd] & & \EE_1' \ar[ld]\\
 & S_1 & 
 }
\]
Here $\EE_1'$ is the pull-back of $\EE_1$ under the Witt vector Frobenius $\mathrm{Fr}\colon S_1 \to S_1$. 
\end{enumerate}
\end{lemma}

\begin{proof}
The properties in the statement are equivalent to $\EE_1 \to S_1$ being the canonical lift of $E_0/k$ in the sense of Serre-Tate, and the proof requires some background concerning Serre-Tate canonical lifts, compare \cite{Ka78}, \cite[Section 2.10]{Hi12}, Appendix by M.V. Nori and V. Srinivas to \cite{MS87}. 

\medskip

Suppose first that a) holds, so this exact sequence splits. Now b) is equivalent to $\EE_1 \to \mathrm{Spec}\, R_1$ being the canonical lift of $E_0/k$ in the sense of Serre-Tate (which is unique) by the Appendix by M.V. Nori and V. Srinivas to \cite{MS87}, Theorem 1) (and its proof, compare also Proposition 1 ibidem). So we have to prove that the splitting of the exact sequence tells us that $\EE_1 \to \mathrm{Spec}\, R_1$ is the canonical lift. Let
\[
T_p E_0 = \varprojlim_n E_0[p^n]
\]
be the Tate module of $E_0$. One knows that for a local artinian $k$-algebra $A$ there is an isomorphism, functorial in $A$ between infinitesimal deformations of $E_0$ over $A$ and $\Z_p$-bilinear maps
\[
q\colon T_p(E_0) \times T_p(E_0) \to 1+\mathfrak{m}_A
\]
where $\mathfrak{m}_A$ is the maximal ideal of $A$ (``Serre-Tate coordinates"), see \cite[Thm. 2.10.5]{Hi12} or \cite[Thm. 2.1]{Ka78}. (Actually it is neater to think of the second factor in the source of the pairing as $T_p(E_0^t)$, the Tate module of the dual abelian variety, which is again isomorphic to $E_0$ in our case, however).

So we need to check that under our hypothesis on the splitting of the sequence, the $q$-pairing is trivial (the canonical lift corresponds to the trivial pairing). In our case, the target $1+\mathfrak{m}_{R_1} = 1+ (p)$ is annihilated by $p$, so the pairing already factors over a pairing 
\[
T_p(E_0)/p \times T_p(E_0)/p \simeq E[p] \times E[p]  \to 1+\mathfrak{m}_{R_1}. 
\]
The construction of $q$ is described \cite[p. 151/152]{Ka78} or \cite[p. 218-221]{Hi12}: in our case, for the pairing to be trivial, we only need to check that the composite
\[
\xymatrix{
T_p(E_0) \ar@{->>}[r] & E_0[p] \ar[r]^{``p"\quad\quad\quad\quad\quad} & \mathrm{Hom}_{\Z_p} (T_p(E_0), 1+\mathfrak{m}_{R_1}) 
}
\]
is trivial, where the homomorphism $``p"$ is defined as follows: for $x\in E_0[p]$, pick a lift $\tilde{x} \in \EE_1(R_1)$ of $x$; then $p\tilde{x}$ does not depend on the chosen lift, and can be identified with an element in $\mathrm{Hom}_{\Z_p} (T_p(E_0), 1+\mathfrak{m}_{R_1})$; however, if the sequence in a) splits we can choose a lift in $\EE_1(R_1)$ of order $p$ whence $p\tilde{x}$ is trivial. 

\medskip

Now suppose that b) holds, the Frobenius lifts. Then again by the Appendix by M.V. Nori and V. Srinivas to \cite{MS87}, Theorem 1), $\EE_1 \to \mathrm{Spec}\, R_1$ is the canonical lift of $E_0/k$. We can extend it to the canonical lift $\EE \to \mathrm{Spec}\, R$ over the entire Witt vectors (not just the first order truncation). But by Serre-Tate theory, lifts of $E_0$ to $\mathrm{Spec}\, R$ correspond to lifts of the $p$-divisible group scheme of torsion points of order a power of $p$ on $E_0$ to $\mathrm{Spec}\, R$, and the Serre-Tate canonical lift is precisely characterised by the fact that that lift splits into the unique lift of the \'{e}tale rank $1$ group and the group of multiplicative type. Thus in particular, the exact sequence in a) splits. 
\end{proof}

It is interesting and necessary for applications to have a way to test when b) of Lemma \ref{lSplittingPTorsion} holds for a concretely given $\EE_1 \to \mathrm{Spec}\, R_1$. By Thm. 1, 3) of the Appendix to \cite{MS87}, if we let $L$ be a degree $1$ line bundle on $E_0$, 
associated to the given origin of $E_0$, it lifts uniquely to a line bundle $\LL$ on $\EE_1 \to \mathrm{Spec}\, R_1$ such that $F^* \LL' \simeq \LL^{\otimes p}$ (where $\LL'$ is the line bundle induced by pull-back by $\LL$ on the Frobenius twist $\EE_1'$). If we use $\LL^{\otimes 3}$ and $(\LL')^{\otimes 3}$ to embed $\EE_1 \to \mathrm{Spec}\, R_1$ and $\EE_1' \to \mathrm{Spec}\, R_1$ into $\P^2_R$ (with homogeneous coordinates $x,y,z$), the Frobenius lift $F$ is given by a triple of homogeneous polynomials of degree $p$ that reduce to $(x^p, y^p, z^p)$ on the central fibre. This gives a way to decide algorithmically if a given lift $\EE_1 \to \mathrm{Spec}\, R_1$ is the canonical lift or not. In fact, it is advantageous to work with all possible lifts at once. 

\medskip

We will illustrate the algorithm in a simple case. Suppose we are given a homogeneous degree $3$ polynomial $e \in \Z[x,y,z]$ such that its reduction $e_p  \in \F_p[x,y,z]$ is the equation of a smooth plane cubic. Also write $f = (x^p, y^p, z^p)$. Then
\[
e (f) - e^p \equiv 0 \, \mod \, p
\]
and thus $e (f) - e^p =  pd$ for some homogeneous polynomial $d$ of degree $3p$. Let $f+pf'$ for $f'$ another triple of homogeneous degree $p$ polynomials be a lift of the Frobenius modulo $p^2$. Taylor expansion gives
\[
e(f+pf') = e(f) + p\,  (\grad e)(f) f' + p^2 \cdot (\mathrm{remainder})
\]
and thus
\[
e(f+pf') - e^p \equiv p (d +  (\grad e)(f) f' ) \, \mod \, p^2 .
\]
Then $f +pf'$ is a lift of the Frobenius if and only if the last expression is a multiple of $e$ modulo $p^2$ meaning one can write it as $p\, c \cdot e$ for another unknown polynomial $c$. Thus we have to solve 
\[
p (d +  (\grad e)(f) f' - c e)\equiv 0 \, \mod \, p^2
\]
or 
\[
d +  (\grad e)(f) f' - c e\equiv 0 \, \mod \, p 
\]
for $f'$ and $c$. This can be done by a Gr\"obner basis calculation.

\section{Bounds for curves with good reduction at two given primes}\label{pBoundsTwoPrimes}

So far we have we have only discussed how to obtain a bound on the image of coprime-to-$p$ torsion 
\[
t (E_1, \pi_1, E_2, \pi_2, p'):= \left| \pi_1\left( E_1[\infty]^{(p')} \right) \cap \pi_2 \left( E_2[\infty]^{(p')} \right) \right|
\] 
Naively, one can take another prime $q\neq p$ satisfying \ref{aBasic} and obtain a bound on $p$-primary torsion
\[
t (E_1, \pi_1, E_2, \pi_2, p):= \left| \pi_1\left( E_1[\infty]^{(p)} \right) \cap \pi_2 \left( E_2[\infty]^{(p)} \right) \right|
\] 
by simply noting that $E_i[\infty]^{(p)}\subseteq E_i[\infty]^{(q')}$.

In \cite{Ray83-1}, Raynaud describes how to combine the two bounds into a total torsion bound. For this, it is equivalent and more convenient to work with the abelian scheme $\sA$ and the curve $\XX$ and instead discuss how to obtain bounds on $t_{E_{1}\times E_{2},X}$; the (geometric) torsion points of $A=E_1 \times E_2$ that lie on $X$. For this we denote by $t_{A,X,p'}$ and $t_{A,X,p}$ the coprime-to-$p$ and $p$-primary torsion lying on $X$ respectively. It will be necessary for us to make the following definition

\begin{definition}\label{dLargeGaloisActions}
    Let $G_{K}$ denote the absolute Galois group of $K$ and let $M$ be a $p$-divisible $G_{K}$ module. We say that the action of $G_{K}$ on $M$ is \textit{large} if for any element $x\in M$ of order at least $p^r$, the size of the orbit $G_{K}\cdot x$ tends to infinity as $r$ tends to infinity.

    Specifically, for any integer $N>0$ there exists an integer $r$ such that for any elements of order $>p^r$ we have that:
    \[|G_{K}\cdot x | > N\]
\end{definition}

Of course we are interested in when $M=A[\infty]^{(p)}$ is the $p$-primary torsion of $A$. It is immediate that the torsion of $A$ may be decomposed into $p$-primary and coprime-to-$p$ parts
\[A[\infty] = A[\infty]^{(p)}\oplus A[\infty]^{(p')}\]
We will combine the bounds on the two types of torsion to obtain a total bound 

\begin{proposition}\label{pTotalTorsionBounds}
    If the action of $G_{K}$ on $A[\infty]^{(p)}$ is large then there exists a constant $c$ such that
    \[t_{A,X} \leq c\]
\end{proposition}

To do this, we are required to strengthen Theorem \ref{tCoarseBounds} and show that the bounds obtained are invariant under translation (cf.\@ \cite[Theorem 4.4.1]{Ray83-1})

\begin{lemma}\label{lCoprimeTorsionBoundsforTranslate}
    Let $a\in A(\bar{K})$. Then if $(X+a) \hookrightarrow A$ denotes the translation of the curve $X$ by $a$ we have
    \[t_{A,X+a,p'} \leq 8p^3\]
\end{lemma}
\begin{proof}
    Of course, what we have shown already is the case $a=0$.
    First consider the case that $a\in A(K)$. Then we may repeat the methods of Theorem \ref{tCoarseBounds}. Namely, we can take $a\in \sA(R)\cong A(K)$ and prove that 
    \[
     \left| \mathrm{im}   \left( p \sA_1 (R_1) \cap (\XX_1+a_1) (R_1) \to \XX_0(k) \right) \right| \le 8p^3.
     \]
     where $a_1 \in \sA_1(R_1)$ is the reduction mod $p^2$ of $a$.  We write 
     \[\Lambda(a_1)=\mathrm{im}   \left( p \sA_1 (R_1) \cap (\XX_1+a_1) (R_1) \to \XX_0(k) \right)\] 
     
     As $k$ is algebraically closed, there exists an element $b_{0}\in \sA(k)$ such that $pb_0=a_0$. Picking any lifting $b_1$ of $b_0$ we have that $a_1=pb_1+c_1$ where $c_1$ is in the kernel of reduction.
     
     \begin{enumerate}[label=(\roman*)]
         \item First we assume that $a_1=pb_1$. Then we see that $\Lambda(pb_1)=\Lambda(0)+pb_0$ and so these sets are of the same cardinality.
         \item Now assume that $a_1=c_1$ is in the kernel of reduction. Then, as in Theorem \ref{tCoarseBounds}, the bundle $V_0$ and curve $\YY_0'$ remain unchanged (as they depend only on the special fibre of the translation) whereas the zero section $\XX_0'$ of $V_0$ changes, corresponding to $\XX_1+c_1$ being a different choice of lifting for $\XX_0 $. However, this new lifting still has normal bundle $\NN_{\XX_0/\sA_0}$ in $\P \bigl( \NN_{\XX_0/\sA} \bigr)$ and so we still have the same bound of $8\delta$ for $\Lambda(c_1)$ using the notation of the Theorem.
     \end{enumerate}

     Now assume that $a\in A(\bar{K}) \setminus A(K)$. Consider the $R$-group scheme $\GG$ of automorphisms of $\XX$ that come from translations by elements of $\sA$. As $\XX$ is smooth and irreducible with fibres of genus $5$, this is a finite group scheme. Moreover, $\XX$ has no infinitesimal automorphisms so $\GG$ is unramified. Thus, $X\neq X+a$ and there exists $\sigma \in G_{K}$ such that
     \[X+a^\sigma \neq X+a\]
     On the other hand, $A[\infty]^{(p')}$ is unramified so that
     \[(X+a)(\bar{K})\cap A[\infty]^{(p')} \subseteq (X+a^\sigma)(\bar{K}) \cap (X+a)(\bar{K})\]
     As this is the intersection of two irreducible curves which are not equal, it can be bounded by $(X+a)\cdot (X+a) = X\cdot X = 8$.
\end{proof}

Similarly, for any $a\in A(\bar{K})$, we repeat the above argument for a second prime $q\neq p$ satisfying Assumption \ref{aBasic} and obtain the following
\begin{lemma}\label{lpPrimaryTorsionBounds}
    For any $a\in A(\bar{K})$, we have that
    \[t_{A,X+a,p} \leq 8q^3\]
\end{lemma}

We can now prove Proposition \ref{pTotalTorsionBounds}

\begin{proof}
    Suppose that $x \in X(\bar{K}) \cap A[\infty]$. Then there exists a unique decomposition $x=x'+x''$ where $x'\in (X-x'')(\bar{K})\cap A[\infty]^{(p')}$ and $x''\in (X-x')(\bar{K})\cap A[\infty]^{(p)}$. As $x'\in A(K)$ we have that $X-x'$ is a $K$-curve and thus
    \[G_{K}\cdot x'' \subseteq (X-x')(\bar{K}) \cap A[\infty]^{(p)}\]
    Thus, using Lemma \ref{lCoprimeTorsionBoundsforTranslate} we have that $|G_{K}\cdot x''| \leq 8q^3$. As the Galois action is large, there exists some $r>0$ such that the order of $x''$ is at most $p^r$ (depending on $q$). As $\textnormal{dim}(A)=2$, there are at most $|A[p^r]|=p^{4r}$ possibilities for $x''$. For each such possibility, the number of possibilities for $x' $ is bounded by
    \[|(X-x'')(\bar{K})\cap A[\infty]^{(p')}| \leq 8p^3\]
    and so in total 
    \[t_{A,X} = |X(\bar{K}) \cap A[\infty]| \leq c:=8p^{4r+3}\]
\end{proof}

From the proposition it is clear that one has to understand when the Galois action on the $p$-primary torsion of an elliptic curve is large. Moreover, to obtain tractable bounds for the whole torsion, one needs to understand how the size of Galois orbits $G_{K}\cdot x$ grow with the order of $x$.

\begin{remark}\label{rGaloisOrbitsandFieldExtensions}
    As $K$ is maximally unramified, $G_{K}$ coincides with the inertia group $I_{K}$. It is immediate that for a $p$-primary torsion point $x\in E(\bar{K})$ we have the following relationship
    \[[K(x):K]=|I_{K}\cdot x|\]
    where $K(x)$ is the smallest extension of $K$ over which the point $x$ is defined.
\end{remark}

The answers to both questions raised are given by the following

\begin{proposition}\label{pGaloisActionsGoodReduction}
    Let $E/K$ be an elliptic curve with good reduction where $K$ is a maximally unramified extension of $\mathbb{Q}_{p}$ with valuation $v$ (normalised so that $v(p)=1$) and valuation ring $R$.
    \begin{enumerate}[label=(\roman*)]
        \item If $E$ has ordinary reduction with Serre-Tate coordinate $\lambda \in 1+pR$ then the action on the $p$-primary torsion of $E$ is large if and only if $\lambda \neq 1$. When $\lambda=1$, i.e.\@ $E$ is the canonical lifting of its reduction in the sense of Serre-Tate, there is a splitting
        \[E[\infty]^{(p)}\cong \mu_{p^\infty} \oplus \mathbb{Q}_p/\mathbb{Z}_p\]
        as $I_K$-modules, where $\mu_{p^\infty}$ is the $p$-divisible group formed from the $p^n$-th roots of unity.

        One may repeat the methods of Theorem \ref{tCoarseBounds} to bound common torsion points of the form 
        \[E[\infty]^{(p')}\oplus \mathbb{Q}_p/\mathbb{Z}_p\]
        as they are unramified and $p$-divisible. Then, using the new decomposition:
        \[E[\infty] \cong \left(E[\infty]^{(p')}\oplus \mathbb{Q}_p/\mathbb{Z}_p\right) \oplus \mu_{p^\infty}\]
        one may now apply Proposition \ref{pTotalTorsionBounds} noting that the action on $\mu_{p^\infty}$ is large.

        \item If $E$ has ordinary reduction with Serre-Tate coordinate $\lambda \neq 1$ then the representation:
        \[I_{K} \rightarrow \textnormal{Aut}(E[p^r])\]
        is given by the short exact sequence:
        \[0 \rightarrow \mu_{p^r} \hookrightarrow E[p^r] \twoheadrightarrow \mathbb{Z}/p^r \mathbb{Z} \rightarrow 0\]
        via the matrix
        \[\begin{bmatrix}
        \chi_{p^r} & \theta_{p^r}\\
        0 & 1
        \end{bmatrix}\]
        where for $\sigma \in I_K$
        \[
            \sigma(\lambda^{\frac{1}{p^r}})=\zeta_{p^r}^{\theta_{p^r}(\sigma)} \cdot \lambda^{\frac{1}{p^r}}
        \]
        where we fix a $p^r$-th root of $\lambda$, $\zeta_{p^r}$ is a primitive $p^r$-th root of unity and $\chi_{p^r}$ is the cyclotomic character.

        Thus, if $x\in E[p^r]$ is of exact order $p^r$ then either:

        \begin{enumerate}
            \item $\lambda^{\frac{1}{p^r}} \in K$, the exact sequence above is split and
            \[
            |I_K \cdot x|= 
            \begin{cases}
                1& \text{if } x\in \mathbb{Z}/p^r \mathbb{Z}\\
                p^r-p^{r-1}              & \text{otherwise}
            \end{cases}
            \]

            \item $\lambda^{\frac{1}{p^r}} \notin K$ and
            \[|I_K \cdot x|= p^{r}-p^{r-1}\]
            
        \end{enumerate}

        \item If $E$ has good supersingular reduction, then the Galois orbit on $p$-primary torsion is large. Moreover, if $x$ is of exact order $p^r$ then
        \[[K(x):K]=p^{2r}-p^{2r-2}\]
    \end{enumerate}
\end{proposition}

\begin{proof}
    \begin{enumerate}[label=(\roman*)]
        \item As discussed in \cite[5.3]{Ray83-1}, to bound the torsion of $A=E_1\times E_2$ lying on $X$ we are required to define a splitting $A[\infty]=T'\oplus T''$ where $T'$ has trivial $G_K$ action and is $p$-divisible and $T''$ has large Galois action. Clearly, it is sufficient to find such a decomposition for each curve $E_i$. In general, this will mean splitting into coprime-to-$p$ and $p$-primary torsion respectively. However, if $E_i$ is a canonical lifting of an ordinary elliptic curve, then there is a height 1, $p$-divisible subgroup $H \subseteq E_i[\infty]^{(p)}$ which is $G_K$-trivial. Writing $T''_i$ for the supplement of $H$ in $E_i[\infty]^{(p)}$ and setting $T'_i=H \oplus E_i[\infty]^{(p')}$, we obtain the desired decomposition for $E_i$.
        
        \item See \cite[Appendix]{Kr97}.
        
        \item The case $r=1$ is \cite[1.9 Prop.\@ 9]{Se72} and this result is extended for all $p^r$-torsion in \cite[Cor.\@ 5.2]{Sm23}.
    \end{enumerate}
\end{proof}

\begin{remark}\label{rValuationofSerreTateCoordinate}
    To obtain bounds on torsion in the ordinary reduction case, it is important to determine the minimal $r$ such that $\lambda^{\frac{1}{p^r}} \in K$ but $\lambda^{\frac{1}{p^{r+1}}} \notin K$. i.e.\@ 
    \[\lambda \in (1+pR)^r=1+p^r R\] 
    and so $v(\lambda-1)=r$.
    
    In \cite{Kr97}[Appendix, Prop.\@ 3], Kraus determines formulae for this value. If $E_0$ is the reduction of $E$, then for the special values $j(E_0)=0$ or $1728$, the valuation is determined by a minimal Weierstrass equation for $E$. Otherwise, we have the equality
    \[v(\lambda-1)=v(j(E)-j_{can}(E_0))\]
    where $j_{can}(E_0)$ is the $j$-invariant of the canonical lifting of $E_0$. The question of determining $j_{can}(E_0))$ as an equation of $j(E_0))$ has been explored in \cite{Fin10} and one may compute an approximation of this $j$-invariant for small primes.
\end{remark}

Let $E/K$ be an elliptic curve with good ordinary reduction. Let $E[p]$ be the $p$-torsion subgroup (viewed as a Galois module), $\EE$ the N\'eron model of $E$ and $\GG$ the $p$-torsion group scheme of its N\'eron model. In the view of the above remarks, it is useful to crystallise the relationship between the finite flat group scheme $\GG$ and the Galois module $\GG_K = E[p]$

\begin{lemma}\label{lFiniteFlatGroupSchemesandGKModules}
    Let $\GG$ and $\HH$ be finite flat group schemes and view $\GG_K$ and $\HH_K$ as $G_K=\mathrm{Gal}(\bar{K}/K)$-modules. Then the natural map is an isomorphism
    \[\mathrm{Hom}_R(\GG,\HH)\longrightarrow \mathrm{Hom}_{G_K}(\GG_K,\HH_K)\]
    i.e.\@ the functor between the category of finite flat group schemes over $R$ to $G_K$-modules is fully faithful.
\end{lemma}
\begin{proof}
    As $K$ has absolute ramification index $e=1$, this is a special case of \cite[4.5 Corollary]{Ta97}.
\end{proof}

\begin{corollary}\label{cExactSequencesofFFGSs}
    Let 
    \[0 \to \GG' \to \GG \to \GG'' \to 0\]
    be an exact sequence of finite flat group schemes over $R$. Then this sequence splits if and only if the corresponding exact sequence of $G_K$-modules is split.
\end{corollary}

\section{Bounds for the cases when one or both of the curves have bad multiplicative reduction}\label{sGoodMult}

We now wish to relax the conditions in Assumption \ref{aBasic} to the effect that we also want to be able to say something about curves that have bad multiplicative reduction at the given prime $p$. We modify that Assumption now.

\begin{definition}\label{dGoodModel}
Suppose we are given a pair $(E, \pi)$ where $E$ is an elliptic curve defined over a number field $K$ and $\pi \colon E \to \P^1$ is a standard projection to $\P^1$, also defined over $K$. We say that a \emph{nice model} for $(E, \pi )$ if the following is true. 

There exists a prime $p$ and a place $v$ of $K$ unramified over $p$ such that, letting $R = W(\overline{\mathbb{F}}_p)$ be the Witt vectors over $k:=\overline{\mathbb{F}}_p$ with fraction field $F= \widehat{K^{\mathrm{ur}}_v} \supset K$, one of the two scenarios below holds. 

\begin{enumerate}
\item
There is an abelian scheme
\[
 \EE \to \mathrm{Spec}\, R , 
\]
with geometric generic fibres equal to (the base change to $F$) of the given elliptic curve $E$ defined over $K$, the central fibre $\EE_0$ is ordinary 
and there is an $R$-morphism
\[
\xymatrix{
\EE  \ar[rd]  \ar[rr]^{\pi_{R}} &  & \P^1_R \ar[ld] \\
 & \mathrm{Spec}\, R & 
}
\]
that is the composition of the quotient morphisms $\EE \to \EE/\iota$, where $\iota$ is the fibrewise involution induced by the inversion maps on the generic fibre, with an $R$-isomorphism $\EE / \iota \simeq \P^1_R$,  and  $\pi_{R}$ induces the given standard projection (base-changed to $F$) on the geometric generic fibre. 
\item
There is a minimal Weierstrass model for $E \times_{\mathrm{Spec}\, K}\mathrm{Spec}\, L$: 
\[
\WW \to \mathrm{Spec}\, R
\]
with nodal rational central fibre $\WW_{0}$ and there is an $R$-morphism 
\[
\xymatrix{
\WW  \ar[rd]  \ar[rr]^{\pi_{R}} &  & \P^1_R \ar[ld] \\
 & \mathrm{Spec}\, R & 
}
\]
that is the composition of the quotient morphism $\WW \to \WW/\iota$, where $\iota$ is again the fibrewise involution induced by the inversion map on the generic fibre, and an $R$-isomorphism $\WW/\iota \simeq \P^1_R$ such that $\pi_{ R}$ induces the given standard projection (base-changed to $F$) on the geometric generic fibre. 
\end{enumerate}
In case (a) we say $(E, \pi)$ has \emph{a nice model with good reduction} and in case (b) \emph{a nice model with bad reduction}.
\end{definition}

\begin{assumption}\label{aGoodBadReduction}
Suppose now that $(E_1, \pi_1)$ and $(E_2, \pi_2)$ are two elliptic curves together with standard projections defined over a number field $K$. We assume that each of them has a nice model 
\[
\xymatrix{
\EE_i  \ar[rd]  \ar[rr]^{\pi_{i, R}} &  & \P^1_R \ar[ld] \\
 & \mathrm{Spec}\, R & 
}
\]
with either good or bad reduction (note that from now on we will also use $\EE$ to denote a Weierstrass model to simplify notation in the sequel). 

We also assume the following: 
\begin{enumerate}
\item
The set of branch points in $\P^1$ of the morphisms induced by the models of the standard projections $\pi_{1, R}$ and $\pi_{2, R}$ on the geometric generic fibres are distinct. In addition, the set of branch points in $\P^1$ of the morphisms induced by $\pi_{1, R}$ and $\pi_{2, R}$ on the normalisations of the special fibres are distinct, too, and disjoint from the images in $\P^1_k$ of the nodes of the special fibres, which are also themselves required to be distinct. 
\item
We will write $\sA \to \mathrm{Spec}\, R$ for the fibre-product of the two given models of the elliptic curves. Moreover, we will write 
\[
\XX = \left( \pi_{1, R} \times \pi_{2, R} \right)^{-1} \left( \Delta_{\P^1_R \times_R \P^1_R} \right) 
\]
for the preimage of the diagonal, which is a proper flat $R$-curve. We will also denote by $\sA^{\circ} \subset \sA$ the largest open subscheme that is smooth over $\mathrm{Spec}\, R$, which is a group scheme, and by $\XX^{\circ}$ the restriction of $\XX$ to $\sA^{\circ}$. In analogy with notation used earlier, we will then also denote by $\sA^{\circ}_1$ the base change of $\sA^{\circ}$ to $\mathrm{Spec}\, R_1$, similarly define $\XX^{\circ}_1$ and denote by $\XX^{\circ}_0$ the central fibre of $\XX^{\circ}$. 

We will then assume that
\[
\mathrm{im}   \left( p \sA_1^{\circ} (R_1) \cap \XX^{\circ}_1 (R_1) \to \XX^{\circ}_0(k) \right)
\]
is finite. 
\end{enumerate}
\end{assumption}

\begin{remark}\label{rImplicationsAssumptions}
We will show below in the sections following Section \ref{sExtension} that under extra assumptions on the nice models, part a) of Assumption \ref{aGoodBadReduction} implies part b), but we are not able to show this in complete generality although it may be true. 
\end{remark}

Remark that a problem to carry over the arguments used in the proof of Theorem \ref{tCoarseBounds} to the present context where we assume $(E_1, \pi_1), (E_2, \pi_2)$ to be subject to Assumption \ref{aGoodBadReduction}, is that, the central fibre of $\sA \to \mathrm{Spec}\, R$ no longer being necessarily nonsingular, it is more subtle to do intersection theory on it. This problem can partially be circumvented by noting that the multiplication by $p$-map is still a rational map on the models that \emph{commutes} with the fibrewise involutions $\iota_1, \iota_2$, hence descends to a rational map from $\P^1_R \times_R \P^1_R$ to itself. In short, it is convenient, in the presence of these involutions, to transfer the entire argument based on the ideas in \cite{Ray83-1} to the product of projective lines over $R$. 

\medskip

We first need to introduce some further notation and definitions, and prove auxiliary results. Everywhere below we suppose from now on that we are in the setup of Assumption \ref{aGoodBadReduction}. 

\begin{definition}\label{dMultiplicationByP}
For $i=1,2$, we denote by 
\[
\xymatrix{
\EE_i \ar@{-->}[rr]^{\mathrm{mult}_{p, i}} \ar[rd] && \EE_i \ar[ld] \\
 & \mathrm{Spec}\, R & 
 }
\]
the multiplication by $p$ map, which is in general only a rational map. It is defined on the largest open subscheme of $\EE_i$ that is smooth over $\mathrm{Spec}\, R$, which is a group scheme. Since the multiplication by $p$ map commutes with the fibrewise involutions given by taking inverses for the group law, we obtain an induced rational map, which we will denote by
\[
\xymatrix{
\P^1_R \ar@{-->}[rr]^{\overline{\mathrm{mult}}_{p, i}} \ar[rd] && \P^1_R \ar[ld] \\
 & \mathrm{Spec}\, R & 
 }
\]
For the sake of brevity, we will write 
\[
\mathrm{mult}_{p} = \mathrm{mult}_{p, 1} \times \mathrm{mult}_{p, 2}
\]
which is thus a rational map
\[
\xymatrix{
\sA \ar@{-->}[rr]^{\mathrm{mult}_{p}} \ar[rd] && \sA \ar[ld] \\
 & \mathrm{Spec}\, R & 
 }
\]
and
\[
\overline{\mathrm{mult}}_{p} =\overline{ \mathrm{mult}}_{p, 1} \times \overline{\mathrm{mult}}_{p, 2}
\]
for 
\[
\xymatrix{
\P^1_R \times_R \P^1_R  \ar@{-->}[rr]^{\overline{\mathrm{mult}}_{p}} \ar[rd] && \P^1_R \times_R \P^1_R \ar[ld] \\
 & \mathrm{Spec}\, R & 
 }
\]
Note that on the central fibres, all the rational maps extend uniquely to morphisms, and we denote these by a suffix $0$, so, for example, we obtain 
\[
\xymatrix{
\P^1_k \times_k \P^1_k  \ar[rr]^{(\overline{\mathrm{mult}}_{p})_0} \ar[rd] && \P^1_k \times_k \P^1_k \ar[ld] \\
 & \mathrm{Spec}\, k & 
 }
\]
and 
\[
\xymatrix{
\sA_0 \ar[rr]^{(\mathrm{mult}_{p})_0} \ar[rd] && \sA_0 \ar[ld] \\
 & \mathrm{Spec}\, k & .
 }
\]
Finally, we denote by $\YY_0 \subset \sA_0$ the reduced preimage of $\XX_0$ under $(\mathrm{mult}_{p})_0$, which is a curve lying over the reduced preimage $\overline{\YY}_0 \subset \P^1_k\times \P^1_k$ of the diagonal  $\Delta_0 \subset \P^1_k \times_k \P^1_k$ under $(\overline{\mathrm{mult}}_{p})_0$.
\end{definition}

\begin{definition}\label{dSpecialPointInP1}
Under our standing Assumption  \ref{aGoodBadReduction} we define a point in $\P^1_k$ to be a \emph{special point} for $(\pi_{i, R})_0$, $i=1, 2$ if it is the image in $\P^1_k$ under 
 \[
\xymatrix{
(\EE_i )_0  \ar[rr]^{(\pi_{i, R})_0} &  & \P^1_k
}
\]
of either a node on $(\EE_i )_0$, or a ramification point of the covering of $\P^1_k$ induced by $(\pi_{i, R})_0$ on the normalisation of $(\EE_i )_0$. 

\medskip

We call a point $(x, y)\in \P^1_k \times \P^1_k$ \emph{special} if $x$ is special for $(\pi_{1, R})_0$ or $y$ is special for $(\pi_{2, R})_0$. 
\end{definition}

\begin{remark}\label{rMeaningAss}
Note that Assumption \ref{aGoodBadReduction} a) precisely amounts to saying that if $(x, y)\in \P^1_k$, then at most one of $x$ and $y$ can be special for a projection $(\pi_{i, R})_0$, but not both at the same time. 
\end{remark}

\begin{definition}\label{dNodalPoint}
We call a point $(x, y)\in \P^1_k \times \P^1_k$ \emph{nodal} if either $x$ or $y$ is the image of a node under $(\pi_{1, R})_0$ or $(\pi_{2, R})_0$. 
\end{definition}

\begin{lemma}\label{lY0Nonsingular}
The curve $\overline{\YY}_0$ is nonsingular. 
\end{lemma}

\begin{proof}
Let $\widetilde{(\EE_i )}_0$ be the normalisation of $(\EE_i )_0$ and 
 \[
\xymatrix{
\widetilde{(\EE_i )}_0  \ar[rr]^{\widetilde{(\pi_{i, R})}_0} &  & \P^1_k
}
\]
the induced double coverings. The preimage $\widetilde{\XX_0}$ of $\XX_0$ in $\widetilde{(\EE_1 )}_0\times \widetilde{(\EE_2 )}_0$ is nonsingular under our assumption that if $(x, y)\in \P^1_k\times  \P^1_k$, then at most one of $x$ and $y$ can be special for a projection $(\pi_{i, R})_0$, but not both at the same time. The preimage $\widetilde{\YY}_0$ of $\YY_0$ under the product of normalisation maps is nonsingular because it is the Frobenius twist of the preimage of $\widetilde{\XX_0}$ under an \'{e}tale map. Now $\overline{\YY}_0$ is a quotient of $\widetilde{\YY}_0$ by an action of $\Z/2\times \Z/2$ with at most $\Z/2$ stabilisers, hence nonsingular.
\end{proof}

\begin{theorem}\label{tMainMixedReductionTypes}
Let $(E_1, \pi_1)$, $(E_2, \pi_2)$ satisfy Assumption \ref{aGoodBadReduction}. Let $\mathbb{M}$ be the set of pairs of torsion points $(t_1, t_2) \in E_1(\overline{K}) \times E_2 (\overline{K})$ with the following properties: 
\begin{enumerate}
\item
$\pi_1 (t_1) =\pi_2 (t_2)$;
\item
$t_1, t_2$ have order coprime to $p$.
\end{enumerate}
Then
\[
\left| \mathbb{M} \right| \le 2p^3 + 2. 
\]
\end{theorem}

\begin{proof}
We consider the set $\mathbb{T}$ of pairs of torsion points $(t_1, t_2) \in E_1(\overline{K}) \times E_2 (\overline{K})$ with the following properties: 
\begin{enumerate}
\item
$\pi_1 (t_1) =\pi_2 (t_2)$;
\item
$t_1, t_2$ have order coprime to $p$;
\item
for $i=1, 2$, $t_i$ does not specialise to a node of $(\EE_i)_0$. 
\end{enumerate}
Elements of the set $\mathbb{T}$ can be identified with sections/$R$-valued points of $\sA^{\circ} \to \mathrm{Spec}\, R$, and their images in the central fibre $(\sA^{\circ})_0\simeq \mathbb{G}_m$ are distinct. 

\medskip

Since $t_1, t_2$ have order coprime to $p$, every element in the set $\mathbb{T}$ is equal to $p$ times another element in that set. Therefore, $\mathbb{T}$ injects into the set
\[
\mathbb{S}= \mathrm{im}   \left( p \sA^{\circ} (R) \cap \XX^{\circ} (R) \to \XX^{\circ}_0(k) \right), 
\]
and also into 
\[
\mathbb{S}_1= \mathrm{im}   \left( p \sA_1^{\circ} (R_1) \cap \XX^{\circ}_1 (R_1) \to \XX^{\circ}_0(k) \right), 
\]
which we have assumed to be finite in Assumption \ref{aGoodBadReduction}, b). 

\medskip

Let now 
\[
\overline{\mathbb{T}} =\left\{ t \in \P^1(\overline{K}) \mid \, \exists (t_1, t_2) \in \mathbb{T} \, :\, t = \pi_1 (t_1)=\pi_2 (t_2) \right\} .
\]
Let $\PP^{\circ}$ be the open subscheme of $\PP = \P^1_R \times_R \P^1_R$ that is the complement of the points of the central fibre $(x, y)$ with $x$ or $y$ nodal. 
Since $\overline{\mathbb{T}}$ is obtained from $\mathbb{T}$ by dividing out by the fibrewise involution, it follows that $\overline{\mathbb{T}}$ injects into the set
\[
\overline{\mathbb{S}}_1= \mathrm{im}   \left( \overline{\mathrm{mult}}_p \left(  \PP_1^{\circ} (R_1) \right) \cap \Delta^{\circ}_1 (R_1) \to \Delta^{\circ}_0(k) \right), 
\]
where $\Delta^{\circ}$ is the complement of the nodal points in $\Delta$. Moreover, the finiteness of $\mathbb{S}_1$, which we assumed, implies the finiteness of $\overline{\mathbb{S}}_1$. To  derive the desired bound for $\overline{\mathbb{S}}_1$, hence for $\overline{\mathbb{T}}$, we now apply Proposition \ref{pComputInt} below and obtain $|\mathbb{T}| \le 2p^3$. To complete the proof, it remains to notice that (a) our assumption that nodes of $(\EE_1)_0$ and $(\EE_2)_0$ map to distinct points in $\P^1_k$ and (b) the fact that torsion points of order coprime to $p$ that do not specialise to a node specialise injectively into the central fibre of the model, taken together imply that $\mathbb{M}$ has at most two more elements than $\mathbb{T}$. This proves the Theorem. 
\end{proof}

\section{Some computations}\label{sComputations}

\newcommand{\Fpbar}{{\overline{\F}_p}}
\newcommand{\Witt}{{R_1}}

\begin{notation}\label{nWitt}
We work over the field $k= \Fpbar$ and denote by $R$ the Witt vectors and by $\Witt$ the Witt vectors of length $2$ over $k$. We write elements of $a \in \Witt$ as $(a_0,a_1)$ with $a_i \in \Fpbar$. The operations in $\Witt$ are
\begin{align*}
	a + b &= (a_0,a_1) + (b_0,b_1) = \left(a_0+b_0, a_1+b_1-\frac{(a_0+b_0)^p-a_0^p-b_0^p}{p}\right) \\
	a \cdot b & = (a_0,a_1) \cdot (b_0,b_1) = \left(a_0b_0,a_0^pb_1+a_1b_0^p\right)
\end{align*}
where the first formula is interpreted formally. 

There is a natural quotient ring homomorphism $\Witt \to \Fpbar$ sending a
Witt vector $a = (a_0,a_1)$ to $a_0$. Notice that due to the nontrivial addition law the inclusion map $\Fpbar \to \Witt$ sending $a_0$ to $(a_0,0)$ is not a ring homomorphism. 
\end{notation}

\begin{lemma}\label{lFormulaWitt}
Let $a = (a_0,a_1)$ represent a point in $\P^1_\Witt$, i.e. $a_i = (a_{i,x} , a_{i,y})$, and let $\varphi \colon  \P^1_\Witt \dashrightarrow \P^1_\Witt$ be a rational map such that $\varphi_0$ is defined in $a_0$ (thus extends to a morphism) and with $(d \varphi)_0 = 0$. 
Then we have
\begin{enumerate}
\item
	$(a_0,0) + (0,a_1)   = (a_0,a_1)$ even though we use the nontrivial addition in the Witt vectors. 
\item
$\varphi ((a_0, a_1)) =\varphi ((a_0, 0))$ is independent of $a_1$.
\end{enumerate}

\end{lemma}

\begin{proof}
For the first formula we compute
\[
	(a_0,0) + (0,a_1) = \left(a_0+0,0+a_1+\frac{(a_0+0)^p-a_0^p-0^p}{p}\right)  = (a_0,a_1)
\]
where $a_0^p = (a_{0,x}^p, a_{0,y}^p)$. 

For part b), using a) and Taylor expansion compute
\begin{align*}
	\varphi((a_0, a_1))
	&= \varphi\bigl((a_0,0) + (0,a_1)\bigr) \\
	&= \varphi\bigl((a_0,0)\bigr) + (0,a_1)d\varphi\bigl((a_0,0)\bigr) \\
	& = \varphi\bigl((a_0,0)\bigr) + \left(0,a_1 \left(d\varphi\bigl((a_0,0)\bigr)\right)_0^p \right) \\
	&= \varphi\bigl((a_0,0)\bigr) 
\end{align*}
\end{proof}

\begin{proposition}\label{pComputInt}
Let $U_1, U_2 \subset \P^1_{R}$ be open subsets that contain the generic point of the central fibre. Let $\psi_i \colon U_i \to U_i$ be morphisms representing rational maps which we will denote by the same letters. Assume that $\psi_i$ has degree $p^2$ and $(d\psi_i)_0=0$.  Consider $\psi = (\psi_1,\psi_2)$ and let $U=U_1\times U_2$. 
Assume that  $\psi_0$, the morphism induced by $\psi$ on the central fibre, is of bidegree $(pd, pe)$. Let furthermore $\Delta_R \subset \P^1_R \times  \P^1_R$ be the diagonal. We denote by $Y_0$  the reduced support of $\psi_0^{-1}(\Delta_0)$ and assume it is nonsingular.

Assume that the number $N$ of points in
\[
\mathrm{im} \left( \psi (U (R_1)) \cap (\Delta \cap U)(\Witt) \to \Delta_0 (k) \right)
\]
is finite. 

Then 
\[
N \le (d+e)p^2.
\]
\end{proposition}

\begin{proof}
If $a = (a_0,a_1)\in U(\Witt)$ is an $\Witt$-valued point such that $\psi(a)=b=(b_0, b_1) \in (\Delta_\Witt \cap U)(\Witt )$, we must have
\[
	\psi_0(a_0) \in \Delta_0 \subset \P^1_\Fpbar \times \P^1_\Fpbar .
\]
Therefore $a_0$ must lie in the support of the preimage of $\Delta_0$ and of course in $U$. Let $F_0=0$ be a bihomogeneous equation defining $Y_0$. Let also 
\[
	Y_R \subset  \P^1_R \times \P^1_R
\]
be the curve defined by the equation $F=0$ where $F$ is obtained from $F_0$ by lifting all coefficients $f_{i,0} \in \Fpbar$ to $(f_{i,0},0, \dots ) \in R$. This is a non-canonical lift of $F_0$, any other lift would also work for our purpose. Since $d \psi_0 = 0$, the morphism $\psi_0$ factors over the Frobenius, and the preimage of $\psi_0$ has multiplicity at least $p$. Therefore $Y_0$ is of bidegree at most $(d,e)$. Similarly $Y_R$ has bidegree at most $(d,e)$.

We now try to find $a_1' \in \Fpbar$ such that
$a' = (a_0,a_1') \in (Y_R \cap U)(\Witt)$ and $\psi(a') = \psi(a)$. Using Taylor expansion we calculate
\[
	0 = F(a') = F\bigl((a_0,a_1')\bigr) = F\bigl((a_0,0)\bigr)+ (0,a_1') dF\bigl((a_0,0)\bigr) 
\]
which can be solved for $a_1'$ if $dF\bigl((a_0,0)\bigr) \not= 0$. This is the case
iff $dF_0(a_0) \not= 0$ which holds because $Y_0$ is smooth in $a_0$.

Using Lemma \ref{lFormulaWitt} we also have
\[
	\psi(a) 
	= \psi\bigl((a_0,a_1) \bigr)
	= \psi\bigl((a_0,0) \bigr)
	= \psi\bigl((a_0,a_1') \bigr) 
	= \psi(a') .
\]



Now consider the scheme-theoretic image $X_{R_1}$ of $\psi \colon Y_{R_1} \cap U \to \P^1_{R_1} \times \P^1_{R_1}$. Recall that by definition this is the smallest closed subscheme of $\P^1_{R_1} \times \P^1_{R_1}$ through which this morphism $\psi$ factors, or equivalently, the closed subscheme defined by the sheaf of ideals
\[
\II = \mathrm{ker} \left(  \OO_{ \P^1_{R_1} \times \P^1_{R_1}} \to \psi_* \OO_{Y_{R_1} \cap U} \right) .
\]
Then $\psi \colon Y_{R_1} \cap U \to X_{R_1}$ is dominant. The closed subscheme $X_{R_1} \subset \P^1_{R_1} \times \P^1_{R_1}$ has no embedded points (otherwise it would not be the smallest closed subscheme through which $\psi$ factors since $Y_{R_1}$ has no embedded points and the preimage under $\psi$ of the pure-one dimensional component of $X_{R_1}$ has to equal $Y_{R_1}\cap U$), and the support of $X_{R_1}$ contains the diagonal $\Delta_0$. Moreover, by \cite[Thm. 11.10.9, Prop. 11.10.1 b)]{EGAIV}, the smallest closed subscheme containing all sections in $\psi ( (Y_{R_1}\cap U)(\Witt))$ equals $X_{R_1}$: indeed, this follows because $(Y_{R_1}\cap U)(\Witt))$ is scheme-theoretically dense in $Y_{R_1} \cap U$ and $\psi \colon Y_{R_1} \cap U \to X_{R_1}$ is dominant. Note that the scheme-theoretic image $X_{R}$ of $\psi \colon Y_{R} \cap U \to \P^1_{R} \times \P^1_{R}$ is flat over $\mathrm{Spec}\, R$ because every irreducible component dominates $\mathrm{Spec}\, R$. 

\medskip

Consider the ideal $I$ defining $X_{R_1}$ and its reduction $I_0$ to $k$.  This reduction defines a curve without embedded points and is therefore generated by a polynomial $G_0 \in I_0$. Since $I \to I_0$ is surjective, we can choose a lift $G$ of $G_0$ in $I$. Let now $G' \in I$ be another polynomial, and $G_0'\in I_0$ its reduction to $k$. Now $I_0$ is generated by $G_0$ and therefore there exists a $L_0$ such that $G'_0 = G_0L_0$. Let $L$ be any lift of $L_0$ to $R_1$. Then 
\[
	G' - LG = pG'' \in I
\]
for some $G''$. Now since $p \not\in I$ this implies $G''_0 \in I_0$. But then $G''_0 = G_0M_0$. If $M$ is any lift of $M_0$ we have that
\[
	G' -(LG+pMG) 
\]
is zero modulo $p^2$. But then $G' = G(L+pM)$ in $R_1$. Therefore $G$ generates the ideal of $X_{R_1}$. 

\medskip

The polynomial $G$ has bidegree at most $(dp^2, ep^2)$ because the curve $G=0$ is contained in the flat limit of $\psi (Y_K)$ where $K= \mathrm{Quot}(R)$. 

We parametrise $\Delta_R$ by $X_0 =T_0^2, X_1=T_0T_1, Y_0= T_1T_0, Y_1=T_1^2$. We put
\[
\widetilde{G}(T_0, T_1) = G (T_0^2, T_0T_1, T_1T_0, T_1^2). 
\]
If $\widetilde{G}$ is not identically zero, then the degree of $\widetilde{G}$ is at most $dp^2 + ep^2$, which gives the bound of the Proposition. 

Assume to the contrary that $\widetilde{G}$ is identically zero. Then the equation of $\Delta_{R_1}$ is a factor of $G$. This is only possible if infinitely many elements of $\psi ( (Y_{R}\cap U)(\Witt))$ lie on $\Delta_{R_1}$. 
In that case, 
\[
\mathrm{im} \left( \psi (U (R_1)) \cap (\Delta \cap U)(\Witt) \to \Delta_0 (k) \right)
\]
is infinite, contrary to our assumption in the statement of the Proposition. 
\end{proof}

\section{Extending the multiplication-by-$p$ map to proper models}\label{sExtension}

In this Section we will keep most of the notation from Section \ref{sGoodMult} except that standard projections from elliptic curves to $\P^1$ will usually be denoted by the letter $\sigma$ instead of $\pi$ from now on since here we will need to give names to a number of structure morphisms to $\mathrm{Spec}\, R$ and will reserve the letter $\pi$ for those. 

We will prove below that the finiteness hypothesis in Assumption \ref{aGoodBadReduction}, b) is implied by Assumption \ref{aGoodBadReduction}, a) under certain extra assumptions on the models of $(E_i, \sigma_i)$, $i=1, 2$. More precisely:

\begin{theorem}\label{tFinitenessGeneralised}
Suppose that $(E_1, \sigma_1)$ and $(E_2, \sigma_2)$ are two elliptic curves together with standard projections defined over a number field $K$. Assume that each of them has a nice model 
\[
\xymatrix{
\WW_1  \ar[rd]  \ar[rr]^{\sigma_{1, R}} &  & \P^1_R \ar[ld] \\
 & \mathrm{Spec}\, R & 
}
\]
respectively
\[
\xymatrix{
\EE_2  \ar[rd]  \ar[rr]^{\sigma_{2, R}} &  & \P^1_R \ar[ld] \\
 & \mathrm{Spec}\, R & 
}
\]
in the sense of Definition \ref{dGoodModel}, and further assume that 
\begin{enumerate}
\item
The model $\pi_{\WW_1}\colon \WW_1 \to \mathrm{Spec}\, R$ is a minimal Weierstrass model with nodal rational central fibre and the elliptic curve $E_1$ over $K$ has a Tate uniformisation $K^*/q^{\mathbb{Z}}$ with a parameter $q\in K^*$ that is a $p$-th power of a uniformiser in $K$.
\item
The model $\pi_{\EE_2}\colon \EE_2 \to \mathrm{Spec}\, R$ is smooth with central fibre an ordinary elliptic curve. 
\end{enumerate}

Then the statement in part a) of Assumption \ref{aGoodBadReduction} implies the finiteness statement in part b). 
\end{theorem}

\begin{remark}\label{rRaynaudProof}
Raynaud in \cite{Ray83-2} describes a method to prove the analogue of Theorem \ref{tFinitenessGeneralised} in the case when both curves have good ordinary reduction. The punchline of the argument is that if the finiteness statement in Assumption \ref{aGoodBadReduction}, b) is false then the relative Frobenius morphism on some smooth proper curve of genus $\ge 2$ would lift infinitesimally to first order, which gives a contradiction. To prove Theorem \ref{tFinitenessGeneralised} we will follow the structure of Raynaud's argument and generalise it to log smooth curves in logarithmic algebraic geometry.  
\end{remark}

The proof of Theorem \ref{tFinitenessGeneralised} needs a number of preparations and will occupy this and the next three sections. The non-liftability of the Frobenius used in Raynaud's argument only holds if one works with proper curves, so as a first step of the proof of Theorem \ref{tFinitenessGeneralised}, we will extend the multiplication by $p$ map for certain elliptic curves with bad multiplicative reduction to some proper models of these curves over $\mathrm{Spec}\, R$.

\medskip

We start by recalling a few general facts about models of elliptic curves needed in the sequel. We retain the previous notation $k=\overline{\mathbb{F}}_p$, $R=W(k)$ the ring of Witt vectors with coefficients in $k$, and $K$ its field of fractions (the completion of the maximal unramified extension of $\mathbb{Q}_p$). Let $E=E_K$ be an elliptic curve defined over $K$. Of course, $E$ being elliptic, it comes with a privileged rational point $o\in E(K)$, the origin for the group-law. Denote by $\pi_{\EE}\colon \EE \to \mathrm{Spec}\, R$ the minimal proper regular model of $E$. The vertical prime divisors of $\EE$ that do not meet $\overline{\{o\}}$ can be contracted to obtain the minimal Weierstrass model of $E$, which we denote by $\pi_{\WW} \colon \WW \to \mathrm{Spec}\, R$ \cite[Thm. 4.35]{Liu02}. 

The largest subschemes $\EE^{\circ}$ and $\WW^{\circ}$ that are smooth over $\mathrm{Spec}\, R$ are $R$-group schemes in a natural way \cite[Prop. 2.7]{De-Ra73}. In particular, for every integer $n$, the multiplication by $n$ maps $[n] \colon \EE^{\circ} \to  \EE^{\circ}$ and $[n] \colon \WW^{\circ} \to  \WW^{\circ}$ are well-defined. 

More precisely, there is a morphism $+ \colon \EE^{\circ} \times_{R} \EE \to \EE$ making $\EE \to \mathrm{Spec}(R)$ into a generalised elliptic curve in the sense of \cite[Def. 1.12]{De-Ra73} or \cite[Def. 1.29]{Sai13}, and the central fibre of $\EE \to \mathrm{Spec}(R)$ is a N\'{e}ron $N$-gon $P_{N, k}$ over $k$ with the action of the smooth locus $P_{N, k}^{\circ} \simeq \mathbb{G}_m^N$ on $P_{N, k}$ being explicitly given as in \cite[\S 1.5, p. 29 ff.]{Sai13}. In a nutshell, $P_{N,k}$ consists of $N$ projective lines, labelled by $\mathbb{Z}/N\mathbb{Z}$, and glued cyclically in such a way that $\infty$ on the $\mathbb{P}^1_k$ with label $i$ gets identified with $0$ on the copy of $\mathbb{P}^1_k$ with label $i+1$, and the action of $P_{N, k}^{\circ} \simeq \mathbb{G}_m^N$ on the N\'{e}ron $N$-gon is given by adding corresponding labels and letting $\mathbb{G}_m$ act naturally on $\mathbb{P}^1_k$ with fixed points $0, \infty$. 

The kernel $K_n=\mathrm{ker}([n])$ of multiplication by $n$ on $\EE^{\circ}$ is an $R$-group scheme that acts on $\EE$ by the above construction. If $n$ divides $N$, it is a finite flat commutative $R$-group scheme, of degree $n^2$, \'{etale} if $n$ is invertible in $R$ \cite[Prop. 1.34, Cor. 1.35]{Sai13}. 

\begin{definition}\label{dAdmissibleFactorisation}
An \emph{admissible factorisation} of the multiplication by $p$ map consists of 
\begin{enumerate}
\item
A projective model $\pi_{\UU} \colon \UU \to \mathrm{Spec}\, R$ of $E$. 
\item
An $R-$morphism $f_p \colon \EE \to \UU$ whose restriction to the generic fibre is the multiplication by $p$ map $[p]\colon E \to E$. 
\item
A flat, projective $R$-scheme $\pi_{\FF} \colon \FF \to \mathrm{Spec}\, R$ with $R$-morphisms
\[
\xymatrix{
\EE \ar[r]^{\alpha_{\EE}} & \FF \ar[r]^{\beta_{\FF}} & \UU
}
\]
such that $f_p = \beta_{\FF} \circ \alpha_{\EE}$ and the morphism $\alpha_{\EE, k} \colon \EE_k \to \FF_k$ induced on the central fibres is the relative Frobenius morphism; in particular, $\FF_k$ is the Frobenius twist of $\EE_k$; and the morphism $\beta_{\FF}$ is \'{e}tale. 
\end{enumerate}
\end{definition}

\begin{proposition}\label{pExistenceFactorisation}
Suppose that the central fibre $\EE_k$ of $\EE\to \mathrm{Spec}(R)$ is either a nonsingular ordinary elliptic curve or that it is a N\'{e}ron $N$-gon with $p$ dividing $N$. Then an admissible factorisation exists. 
\end{proposition}

\begin{proof}
In the case when $\EE_k$ is a nonsingular ordinary elliptic, this has already been observed in \cite[p. 5/6]{Ray83-2}: indeed, in this case, we can let $\UU =\EE$ and denoting by $K_p$ the kernel of multiplication by $p$ on $\EE^{\circ}$, $K_p^{\circ}$ its identity component, one can define $\FF := \EE/ K_p^{\circ}$ (the quotient of $\EE$ by the action of the finite group scheme $K_p^{\circ}$). 

So we consider the case when $\EE_k=P_{N, k}$ is a N\'{e}ron $N$-gon in the sequel, with $N=p\cdot m$. The main point now is that, since $p$ divides $N$, the kernel of multiplication by $p$, $K_p$, is a \emph{finite flat} $R$-group subscheme of $\EE^{\circ}$ that acts on $\EE$ by restricting the morphism $+ \colon \EE^{\circ} \times_{R} \EE \to \EE$ to $K_p$. Then by \cite{Ray66} or \cite[Chapter 4, Thm. 4.16, p. 55]{EGM}, we obtain that there exists a geometric quotient $\UU := \EE / K_p$ that is an integral, projective, flat $R$-scheme (by part (i) of the Theorem in loc. cit.), and the quotient morphism $\EE \to \UU$ is given by multiplication by $p$ on the generic fibre (for example by \cite[Chapter 4, Thm. 4.16, part (ii), p. 55]{EGM}, compatibility with flat base change).  

We can perform the same construction with any finite flat $R$-subscheme of $K_p$. Now by \cite[A.1.2, IV-31, (1)]{Se88} $K_p$ sits in an exact sequence of finite flat $R$-group schemes
\begin{gather}\label{eTorsionE1}
0 \to \bbmu_p  \to K_p \to \mathbb{Z}/p\mathbb{Z} \to 0.
\end{gather}
Here $K_p^{\circ}=\bbmu_p$ is the connected component of the identity, and the quotient is \'{e}tale. If we let $\FF = \EE / \bbmu_p$ all the desired properties of the proposition hold. 
\end{proof}

\begin{remark}\label{rQuotientSings}
\'{E}tale locally around a singular point of the special fibre, $\EE$ is isomorphic to the subscheme of $\mathbb{A}^2_R$ given by $XY-\pi =0$ cf. \cite[I. Thm. 5.3]{De-Ra73} for a uniformiser $\pi$ of $R$. The $\bbmu_p=\mathrm{Spec}\, R[T]/(T^p -1)$-action is given locally around this $\bbmu_p$-fixed point by
\begin{align*}
R[X, Y](XY-\pi) & \to R[T]/(T^p-1) \otimes_R R[X, Y](XY-\pi)\\
X & \mapsto T \otimes X\\
Y & \mapsto T^{-1} \otimes Y
\end{align*}
and the quotient $\FF$ (and hence also $\UU$) can be described \'{e}tale locally around the image of that singular point as $UV - \pi^p=0$ in $\mathbb{A}^2_R$. 
\end{remark}

\section{The geometry of preimages of the diagonal under certain covering maps}\label{sGeometryCovering}

We work over $k=\overline{\mathbb{F}}_p$ in this section, assume $p\neq 2$ from now, and consider 
\begin{enumerate}
\item
A nodal rational cubic $C_0$ with a degree $2$ covering $\sigma \colon C_0 \to \mathbb{P}^1$. Precomposing with the normalisation morphism of $C_0$ we get a degree $2$ covering $\widetilde{\sigma} \colon \widetilde{C}_0 \simeq \mathbb{P}^1 \to \mathbb{P}^1$ branched in two points $p_1, p_2\in \mathbb{P}^1$. Let $\gamma \colon P_{N, k} \to C_0$ be the \'{e}tale $N:1$ cover of $C_0$ by the N\'{e}ron $N$-gon. 
\item
An elliptic curve $E_0$ over $k$ with a double covering $\tau \colon E_0 \to \mathbb{P}^1$ branched in four points $q_1, \dots , q_4$, identifying a point and its inverse for the group law on $E_0$ in each fibre. We assume the sets $\{q_1, \dots , q_4\}$ and $\{ p_1, p_2\}$ are disjoint. We also assume each $q_i$ is different from the image of the node on $C_0$ under $\sigma$. 
\end{enumerate}

Let $\Delta \subset \mathbb{P}^1\times \mathbb{P}^1$ be the diagonal. 
We wish to determine the geometry of the preimage curve 
\[
\Gamma = \left( (\gamma\circ \sigma) \times \tau \right)^{-1} (\Delta ) \subset P_{N, k} \times E_0. 
\]
This can be reduced to determining the geometry of 
\[
\overline{\Gamma}= \left( \widetilde{\sigma} \times \tau \right)^{-1} (\Delta ) \subset \widetilde{C}_0 \times E_0.
\]
It is easy to see that since the sets of branch points for $\widetilde{\sigma}$ and $\tau$ are disjoint, the curve $\overline{\Gamma}$ is nonsingular and irreducible (nonsingularity can be checked \'{e}tale/analytically locally, and irreducibility holds because, again looking \'{e}tale locally, one sees that if $\overline{\Gamma}$ were reducible, it would split into two components permuted by the covering group $\mathbb{Z}/2\mathbb{Z} \times \mathbb{Z}/2\mathbb{Z}$, but again since the sets of branch points for $\widetilde{\sigma}$ and $\tau$ are disjoint, a local argument shows that no subgroup $\mathbb{Z}/2\mathbb{Z}$ of $\mathbb{Z}/2\mathbb{Z} \times \mathbb{Z}/2\mathbb{Z}$ acts trivially on the set of components, a contradiction). 

Denoting by $F_1, F_2$ a fibre of the first and second projection of $\widetilde{C}_0 \times E_0$ onto its factors, we see that $\overline{\Gamma}$ is numerically equivalent to $2F_1 + 2F_2$. The canonical class $K_S$ of $S:=\widetilde{C}_0 \times E_0$ being $-2F_1$, we get for the genus of $\overline{\Gamma}$
\[
g (\overline{\Gamma}) = \frac{1}{2} \overline{\Gamma}\cdot (\overline{\Gamma} + K_S) +1 = 3. 
\]
Let $\nu_1, \nu_2\in \widetilde{C}_0$ be the points mapping to the node of $C_0$ under the normalisation morphism; since we assumed that each $q_i$ is different from the image of the node on $C_0$ under $\sigma$, it follows that $\overline{\Gamma}$ intersects $\{\nu_1\} \times E_0$ and $\{\nu_2\} \times E_0$ transversely in $S$. Thus in summary we get

\begin{proposition}\label{pGeometryCurve}
The curve $\Gamma$ is a connected curve with $N$ connected components each of which is a nonsingular curve of genus $3$. These connected components intersect in points that are nodes on $\Gamma$. 
\end{proposition}

\section{The connection to torsion points}\label{sConnectionTorsion}

Suppose now that we are given two elliptic curves $E_1, E_2$ over $K$ with standard projections $\sigma_i$ as in Theorem \ref{tFinitenessGeneralised}. Then, with the hypotheses and notation of Section \ref{sExtension} (adding indices $1$ and $2$ to $\EE$ etc.), the minimal proper regular model $\EE_1 \to \mathrm{Spec}\, R$ of $E_1$ has central fibre a N\'{e}ron $p$-gon, whereas $\EE_2 \to \mathrm{Spec}\, R$ has central fibre an ordinary reduction elliptic curve. Proposition \ref{pExistenceFactorisation} and its proof then produce admissible factorisations
\[
\xymatrix{
\EE_1 \ar[r]^{\alpha_{\EE_1}} & \FF_1 \ar[r]^{\beta_{\FF_1}} & \UU_1, \\
\EE_2 \ar[r]^{\alpha_{\EE_2}} & \FF_2 \ar[r]^{\beta_{\FF_2}} & \UU_2 
}
\]
and $\UU_1 =\WW_1$ is the minimal Weierstrass model and $\UU_2 =\EE_2$. 

We put
\begin{enumerate}
\item
$\sA = \EE_1 \times_{\mathrm{Spec}\, R} \EE_2$, $\BB = \FF_1 \times_{\mathrm{Spec}\, R} \FF_2$, $\CC = \UU_1 \times_{\mathrm{Spec}\, R} \UU_2$ with structural morphisms to $\mathrm{Spec}\, R$ denoted by $\pi_{\sA}, \pi_{\BB}, \pi_{\CC}$. (Note that this is a slight departure from the notation used in Section \ref{sGoodMult} inasmuch there the letter $\sA$ was used for what is denoted by $\CC$ here and in the sequel. However, the notation we now adopt will make the following arguments, somewhat heavy on notation anyway, more transparent and readable we hope). 

\item
$\alpha = \alpha_{\EE_1} \times_R \alpha_{\EE_2}$, $\beta = \beta_{\FF_1} \times_R \beta_{\FF_2}$. 
\end{enumerate}
So we get a sequence of morphisms of $R$-schemes
\[
\xymatrix{
\sA \ar[r]^{\alpha} & \BB \ar[r]^{\beta} & \CC
}
\]
and $\beta\circ \alpha$ restricted to the generic fibre is the multiplication by $p$ map on the abelian surface $E_1 \times E_2$, $\beta$ is \'{e}tale, and $\alpha$ restricts to the relative Frobenius on the central fibre of $\pi_{\sA}\colon \sA \to \mathrm{Spec}\, R$.

By the assumptions made in Theorem \ref{tFinitenessGeneralised}, we are also given standard double covers
\[
\sigma_{E_i}\colon E_i \to \mathbb{P}^1_K, \: i=1, 2, 
\]
extending to double covers
\[
\sigma_{\UU_i}\colon \UU_i \to \mathbb{P}^1_R, \: i=1, 2
\]
of the minimal Weierstrass models. 

Given some model over $\mathrm{Spec}\, R$, we denote the largest open subscheme of it that is smooth over $\mathrm{Spec}\, R$ by an upper $\circ$, such as in $\sA^{\circ}$ for example.

\medskip

With $\Delta_R \subset \mathbb{P}^1_R \times \mathbb{P}^1_R$ the diagonal, we introduce the further notation
\[
\Gamma_{\CC, R} = \left( \sigma_{\UU_1}\times_R\sigma_{\UU_2} \right)^{-1} \left( \Delta_R \right) \subset \CC , \quad \Gamma_{\BB, R} := \beta^{-1}\left( \Gamma_{\CC, R} \right) \subset \BB ,
\]
and denote by $\Gamma_{\CC, k} \subset \CC_k$, $\Gamma_{\BB, k} \subset \BB_k$ the special fibres of these $R$-schemes. Furthermore, we write $\overline{\Gamma}_{\sA, k}$ for the \emph{reduced} preimage of $\Gamma_{\BB, k}$ under $\alpha_k \colon \sA_k \to \BB_k$. 

\medskip

Note that $\Gamma_{\CC, R}$ was denoted by $\XX$ previously in Section \ref{sGoodMult}. 

\medskip

We also denote $R_j :=R/ p^{j+1} R$, and by an upper index in round brackets the pull back of the various $R$-schemes to $\mathrm{Spec}\, R_j$. So, for example, $(\sA^{\circ})^{(1)} (R_1)$ are $R_1$-valued points of $(\sA^{\circ})^{(1)}$, the pull back of $\sA^{\circ}$ to $\mathrm{Spec}\, (R/p^2 R)$. 

\medskip

Let $\Sigma$ be the set of points in $(\sA^{\circ})^{(1)} (R_1)$ that lift points of $\overline{\Gamma}_{\sA, k}^{\circ}$, and set $\Lambda = \alpha^{(1)} \left( \Sigma \right) \subset (\BB^{\circ})^{(1)} (R_1)$. Then $p\Sigma = \beta^{(1)}(\Lambda)$ is the subset of points in $p \left( (\sA^{\circ})^{(1)} (R_1) \right) $ that lift points of $\Gamma_{\CC, k}^{\circ}$. Moreover,

\begin{gather*}
\beta^{(1)} \left( \Lambda \cap (\Gamma_{\BB, R}^{\circ})^{(1)} (R_1)  \right) = \beta^{(1)} \left( \Lambda \cap (\beta^{(1)})^{-1}\left( (\Gamma_{\CC, R}^{\circ})^{(1)} (R_1) \right)  \right)\\
= \beta^{(1)} \left( \Lambda \right) \cap \left( (\Gamma_{\CC, R}^{\circ})^{(1)} (R_1) \right) = p\Sigma \cap \left( (\Gamma_{\CC, R}^{\circ})^{(1)} (R_1) \right). 
\end{gather*} 

Thus we obtain

\begin{lemma}\label{lFiniteCheck}
If the image of $\Lambda \cap (\Gamma_{\BB, R}^{\circ})^{(1)} (R_1)$ in $\Gamma_{\BB, k}^{\circ}(k)$ is finite, then the image of $p\Sigma \cap \left( (\Gamma_{\CC, R}^{\circ})^{(1)} (R_1) \right)$ in $\Gamma_{\CC, k}^{\circ}(k)$ is finite. 
\end{lemma}

In Theorem \ref{tFinitenessGeneralised} we assumed that the elliptic curve $E_1$ over $K$ has a Tate uniformisation $K^*/q^{\mathbb{Z}}$ with a parameter $q\in K^*$ that is a $p$-th power of a uniformiser $\varpi$ in $K$: $q=\varpi^p$. We will now use that assumption to prove

\begin{proposition}\label{pRotationalSymmetry}
If the image of $\Lambda \cap (\Gamma_{\BB, R}^{\circ})^{(1)} (R_1)$ in $\Gamma_{\BB, k}^{\circ}(k)$ is infinite, then this image is infinite in every irreducible component of $\Gamma_{\BB, k}$. 
\end{proposition}

\begin{proof}
This follows from the rotational symmetry of the situation, more precisely: choose a $p$-torsion point $t$ of $E_1(\overline{K})$ such that $(t, \mathrm{id}_{E_2})$ defines an $R$-valued point $x_t$ of $\sA^{\circ}$ intersecting the central fibre $\sA^{\circ}_k$ in a point not lying on the identity component of $\sA^{\circ}_k$. Such $t$ exist, for example, the torsion point $t$ corresponding to $\varpi \in K^*$ under the Tate uniformisation. 
Also $x_t$ induces an $R_1$-valued point of $(\sA^{\circ})^{(1)}$ which we denote by the same symbol.

Suppose now given an $R_1$-valued point $y$ in $\Lambda \cap (\Gamma_{\BB, R}^{\circ})^{(1)} (R_1)$ that specialises to a point on a certain component of $\Gamma_{\BB, k}^{\circ}(k)$. Adding the $R_1$-valued point $\alpha^{(1)}(x_t)$ to $y$ multiple times for the structure of $\BB^{\circ}$ as an $R$-group scheme, we obtain from $y$ points in $\Lambda \cap (\Gamma_{\BB, R}^{\circ})^{(1)} (R_1)$ specialising into points on all the other components. 
\end{proof}

\section{Log deformation theory and Frobenius liftings}\label{sLogFrobenius} 

We start by noticing that $\pi_{\sA}\colon \sA \to \mathrm{Spec}\, R$ becomes log smooth if we endow $\sA$ with the divisorial log structure determined by the central fibre $\sA_k \subset \sA$ and $\mathrm{Spec}\, R$ with the divisorial log structure given by its closed point \cite[Thm. 4.1]{Kato96} or \cite[IV., Thm. 3.1.18]{Ogus18}. We denote the resulting morphism of log schemes
\[
\pi_{\sA}^{\dagger} \colon \sA^{\dagger} \to (\mathrm{Spec}\, R)^{\dagger}, 
\]
and will adhere to the same practice of denoting log schemes by an added dagger in other instances below. 

In fact, $\pi_{\BB}\colon \BB \to \mathrm{Spec}\, R$ and $\pi_{\CC}\colon \CC \to \mathrm{Spec}\, R$ also become log smooth over $(\mathrm{Spec}\, R)^{\dagger}$ if we endow the total spaces with the divisorial log structures determined by the central fibres, and $\alpha$, $\beta$ naturally determine morphisms of log schemes, which we denote $\alpha^{\dagger}$, $\beta^{\dagger}$; indeed, it suffices to check this \'{e}tale locally around singular points of the central fibres where these fibrations are given by
\[
(xy - \pi^p =0 ) \subset \mathbb{A}_R^2
\]
(where we denote a uniformiser of $R$ by $\pi$). By \cite[Ex. 3.27, 3.28]{Gross11}, the log morphism down to $(\mathrm{Spec}\, R)^{\dagger}$ can be described, using charts, by the diagram
\[
\xymatrix{
S_p = \mathbb{N}^2 \oplus_{\mathbb{N}} \mathbb{N} \ar[r] & R [x,y]/(xy - \pi^p) \\
\mathbb{N}\ar[u] \ar[r] & R \ar[u] 
}
\]
where $S_p$ is the submonoid of $\mathbb{N}^2 \oplus \mathbb{N}$ generated by $$\alpha_1 = ((1,0),0), \alpha_2 = ((0,1),0), \varrho = ((0,0),1)$$ with one relation $\alpha_1 + \alpha_2 = p\varrho$, and denoting by $1$ the standard generator of $\mathbb{N}$ (the copy in the left hand lower corner in the diagram), the maps are given as follows:
\begin{enumerate}
\item
$\mathbb{N} \to S_p$ maps $1\mapsto \varrho$;
\item
$S_p \to R[x,y]/(xy-\pi^p)$ sends $\alpha_1\mapsto x, \alpha_2\mapsto y, \varrho \mapsto \pi$;
\item
$\mathbb{N}\mapsto R$ satisfies $1\mapsto \pi, 0 \mapsto 1$;
\item
$R \to R [x,y]/(xy - \pi^p)$ is the natural inclusion. 
\end{enumerate}
Thus the toroidal characterisation of log smoothness \cite[Thm. 4.1]{Kato96} applies. 

\medskip

Restricting the log structure from $\pi_{\sA}^{\dagger}$ to the subscheme $\overline{\Gamma}_{\sA, k}$, we obtain a log scheme $\overline{\Gamma}_{\sA, k}^{\dagger}$ log smooth over the standard log point $(\mathrm{Spec}\, k)^{\dagger}$, which one checks \'{e}tale locally as before.

Using \cite[Prop. 3.40, 3.28]{Gross11}, we can lift $\overline{\Gamma}_{\sA, k}^{\dagger}$ to a log smooth curve $\ZZ^{\dagger} \to (\mathrm{Spec}\, R)^{\dagger}$. Note that $\Gamma_{\BB, R}\to \mathrm{Spec}\, R$ also becomes log-smooth if we endow total space and base with the divisorial log structures determined by the central fibre and marked point, yielding $\Gamma_{\BB, R}^{\dagger} \to (\mathrm{Spec}\, R)^{\dagger}$. 

Our goal now is to show that under the assumptions of Proposition \ref{pRotationalSymmetry}, the morphism $\alpha$ induces a first order infinitesimal lifting of the relative Frobenius 
\[
\xymatrix{
(\ZZ^{(1)})^{\dagger} \ar[rr]^{\Phi}\ar[rd]  &  & (\Gamma_{\BB, R}^{(1)})^{\dagger} \ar[ld] \\
 & (\mathrm{Spec}\, R_1)^{\dagger} & 
}
\]
which is a morphism of log schemes that are log smooth over $(\mathrm{Spec}\, R_1)^{\dagger}$. This will yield a contradiction as in \cite[Lemma I.5.4]{Ray83-2}, using a log version of the Cartier operator and log differential forms. Then by Proposition \ref{pRotationalSymmetry} and Lemma \ref{lFiniteCheck} we conclude that the conclusion of Theorem \ref{tFinitenessGeneralised} holds. 

\medskip

To start we have

\begin{lemma}\label{lLocalLifting}
There exists a canonical morphism of log schemes
\[
\xymatrix{
(\ZZ^{(1)})^{\dagger} \ar[rr]^{\varphi}\ar[rd]  &  & (\BB^{(1)})^{\dagger} \ar[ld] \\
 & (\mathrm{Spec}\, R_1)^{\dagger} & 
}
\]
that lifts $\alpha_k^{\dagger}$ on $\overline{\Gamma}_{\sA, k}^{\dagger}$ and satisfies $\varphi ((\ZZ^{\circ})^{(1)} (R_1)) =\Lambda$. 
\end{lemma}

\begin{proof}
We wish to mimic \cite[Lemma I.5.2]{Ray83-2} in the present log setting. Since $(\sA^{(1)})^{\dagger} \to (\mathrm{Spec}\, R_1)^{\dagger}$ is log smooth, we can lift the inclusion of $\overline{\Gamma}_{\sA, k}^{\dagger}$ into the central fibre \'{e}tale locally, using the categorical characterisation, or rather definition, of log smoothness \cite[Definition 3.7]{Kato96}. Two different such lifts differ by a derivation \cite[Section. 1.1]{Ser06}, but since the differential of $\alpha_k \colon \sA_k \to \BB_k$ is zero, we get a well-defined map to $\BB^{(1)}$ if we compose with $\alpha^{(1)}$.  Since morphisms can be defined \'{e}tale-locally on the source, these local lifts glue to a morphism $\varphi\colon (\ZZ^{(1)})^{\dagger}  \to (\BB^{(1)})^{\dagger}$. The property $\varphi ((\ZZ^{\circ})^{(1)} (R_1)) =\Lambda$ is clear by construction. 
\end{proof}

\begin{lemma}\label{lFrobeniusLifting}
Suppose that the the assumptions of Proposition \ref{pRotationalSymmetry} are satisfied, in particular, the image of $\Lambda \cap (\Gamma_{\BB, R}^{\circ})^{(1)} (R_1)$ in $\Gamma_{\BB, k}^{\circ}(k)$ is infinite. Then the morphism $\alpha$ induces a first order infinitesimal lifting of the relative Frobenius 
\[
\xymatrix{
(\ZZ^{(1)})^{\dagger} \ar[rr]^{\Phi}\ar[rd]  &  & (\Gamma_{\BB, R}^{(1)})^{\dagger} \ar[ld] \\
 & (\mathrm{Spec}\, R_1)^{\dagger} & 
}
\]
which is a morphism of log schemes that are log smooth over $(\mathrm{Spec}\, R_1)^{\dagger}$. 
\end{lemma}

\begin{proof}
This is the analogue in the log setting of \cite[Lemma I.5.3]{Ray83-2}. We wish to show that the morphism $\phi$ of Lemma \ref{lLocalLifting} factors through the closed subscheme $(\Gamma_{\BB, R}^{(1)})^{\dagger}$ in $(\BB^{(1)})^{\dagger}$. We denote by $\widetilde{\ZZ^{(1)}}$ the closed subscheme of $\ZZ^{(1)}$ that we obtain when we pull back $\Gamma_{\BB, R}^{(1)}$ have via $\varphi$. We want to show that $\widetilde{\ZZ^{(1)}} = \ZZ^{(1)}$ and for that it suffices to show that $\widetilde{\ZZ^{(1)}}$ is schematically dense in $\ZZ^{(1)}$. Since we assume that the image of $\Lambda \cap (\Gamma_{\BB, R}^{\circ})^{(1)} (R_1)$ in $\Gamma_{\BB, k}^{\circ}(k)$ is infinite, this image is infinite in every irreducible component of $\Gamma_{\BB, k}$ by Proposition \ref{pRotationalSymmetry}. Therefore there is a set of sections in $\widetilde{\ZZ^{(1)}} (R_1)$ with Zariski dense image in every irreducible component of the special fibre of $\ZZ$, which is $\overline{\Gamma}_{\sA, k}$. Then $\widetilde{\ZZ^{(1)}}$ is schematically dense in $\ZZ^{(1)}$ by \cite[11.10.9]{EGAIV}. 
\end{proof}

We now want to show that there is no lifting of Frobenius as in Lemma \ref{lFrobeniusLifting}, showing the finiteness of the image of $p\Sigma \cap \left( (\Gamma_{\CC, R}^{\circ})^{(1)} (R_1) \right)$ in $\Gamma_{\CC, k}^{\circ}(k)$ under our assumptions. 

\begin{lemma}\label{lFrobDoesNotLift}
Suppose $C^{\dagger} \to (\mathrm{Spec}\, R_1)^{\dagger}$ and $D^{\dagger} \to (\mathrm{Spec}\, R_1)^{\dagger}$ are log smooth curves, and denote by $C_0^{\dagger} \to (\mathrm{Spec}\, k)^{\dagger}$ and $D_0^{\dagger} \to (\mathrm{Spec}\, k)^{\dagger}$ their central fibres, which are the base changes to the standard log point. Assume $D_0$ is the Frobenius twist of $C_0$. Suppose there is a nonsingular component $D_0'$ of $D_0$ on which
\[
\omega_{D_0}\left( \sum_{i=1}^n x_i\right) 
\]
has positive degree, where $x_1, \dots , x_n$ are the double points of $D_0$ or log marked points lying on $D_0'$ as in \cite[Example 3.26 and Examples 3.36 (6)]{Gross11}. 
Suppose also that on the corresponding component $C_0'$ of $C_0$ there is a matching number $y_1, \dots , y_n$ of double points or log marked points. 
Then there is no 
first order infinitesimal lifting of the relative Frobenius  
\[
\xymatrix{
C^{\dagger} \ar[rr]^{\Phi}\ar[rd]  &  & D^{\dagger} \ar[ld] \\
 & (\mathrm{Spec}\, R_1)^{\dagger} & .
}
\]
\end{lemma}

\begin{proof}
We argue similarly to \cite[Lemma I.5.4]{Ray83-2}. First, since $C^{\dagger} \to (\mathrm{Spec}\, R_1)^{\dagger}$ and $D^{\dagger} \to (\mathrm{Spec}\, R_1)^{\dagger}$ are log smooth curves, the sheaves of log differentials $\Omega^1_{C^{\dagger} / (\mathrm{Spec}\, R_1)^{\dagger}}$ and $\Omega^1_{D^{\dagger} / (\mathrm{Spec}\, R_1)^{\dagger}}$ are locally free of rank $1$, and in any event we have a natural functorial morphism of these lines bundles
\[
\Phi^* \colon \Phi^* \Omega^1_{D^{\dagger} / (\mathrm{Spec}\, R_1)^{\dagger}} \to \Omega^1_{C^{\dagger} / (\mathrm{Spec}\, R_1)^{\dagger}},
\]
cf. \cite[p. 115, 116]{Gross11}. Since the differential of the restriction of $\Phi$ to the central fibre, $\Phi_0$, is zero, this morphism of line bundles $\Phi^*$ factors through $p \Omega^1_{C^{\dagger} / (\mathrm{Spec}\, R_1)^{\dagger}}$ and dividing by $p$, we get a morphism 
\[
\tau \colon \Phi^*_0  \Omega^1_{D_0^{\dagger} / (\mathrm{Spec}\, k)^{\dagger}} \to \Omega^1_{C_0^{\dagger} / (\mathrm{Spec}\, k)^{\dagger}}
\]
or, what is the same thing by adjunction, a morphism
\[
\tau' \colon \Omega^1_{D_0^{\dagger} / (\mathrm{Spec}\, k)^{\dagger}} \to (\Phi_0)_*\Omega^1_{C_0^{\dagger} / (\mathrm{Spec}\, k)^{\dagger}}. 
\]
Now both of these maps are nonzero because away from the log marked or double points of $C_0$, the Cartier operator furnishes an inverse to $\tau'$ as in \cite[p. 8, proof of Lemma I.5.4]{Ray83-2}. But now \cite[Examples 3.36 (6)]{Gross11} tells us that $\Omega^1_{D_0^{\dagger} / (\mathrm{Spec}\, k)^{\dagger}}$ restricted to $D_0'$ is nothing but $\omega_{D_0}\left( \sum_{i=1}^n x_i\right)$, which we assumed to have positive degree $d>0$, say. Then $\Phi^*_0  \Omega^1_{D_0^{\dagger} / (\mathrm{Spec}\, k)^{\dagger}}$ will have degree $pd$ on the corresponding component $C_0'$ of $C_0$ (which is just a Frobenius twist of $D_0$). This is a contradiction because $\Omega^1_{C_0^{\dagger} / (\mathrm{Spec}\, k)^{\dagger}}$ has the same degree $d$ when restricted to $C_0'$, but there cannot be a nonzero morphism from a line bundle of degree $pd$ to one of degree $d$ for $d>0$. 
\end{proof}

We can now finally put everything together and give the 

\begin{proof}[Proof of Theorem \ref{tFinitenessGeneralised}]
If the conclusion of the Theorem is wrong, then in particular, by Proposition \ref{pRotationalSymmetry} and Lemma \ref{lFiniteCheck}, we are in the case when Lemma \ref{lFrobeniusLifting} applies. But by Proposition \ref{pGeometryCurve}, each component of $\Gamma_{\BB , k}$ is a nonsingular curve of genus $3$. This is a contradiction to Lemma \ref{lFrobDoesNotLift}. 
\end{proof}

\end{document}